%% file: cut_cells.tex
\documentclass[]{scrartcl}

\usepackage[utf8]{inputenc}
\usepackage[USenglish]{babel}
\usepackage{csquotes}

\usepackage[a4paper,top=27mm,bottom=20mm,inner=25mm,outer=20mm]{geometry}

\usepackage[%
  backend=bibtex,bibencoding=ascii,
  style=numeric-comp,
  giveninits=true, uniquename=init, 
  natbib=true,
  url=true,
  doi=true,
  isbn=false,
  backref=false,
  maxnames=99,
  ]{biblatex}
\usepackage{amsmath}
\allowdisplaybreaks
\numberwithin{equation}{section}
\usepackage{amssymb}
\usepackage{commath}
\usepackage{mathtools}
\usepackage{bbm}
\usepackage{nicefrac}
\usepackage{subdepth}

\usepackage{algorithm}
\usepackage{algpseudocode}

\usepackage{tikz}
\usetikzlibrary[patterns]

\usepackage{amsthm}
\usepackage{thmtools}
\usepackage{etoolbox}
\usepackage{multirow}

\usepackage{breakurl}

\makeatletter
\patchcmd{\thmt@setheadstyle}
 {\bgroup\thmt@space}
 {\thmt@space}
 {}{}
\patchcmd{\thmt@setheadstyle}
 {\egroup\fi}
 {\fi}
 {}{}
\makeatother
\declaretheoremstyle[
  bodyfont=\normalfont\itshape,
  headformat=\NAME\ \NUMBER\NOTE,
]{myplain}
\declaretheoremstyle[
  headformat=\NAME\ \NUMBER\NOTE,
]{mydefinition}
\newcommand{\envqed}{{\lower-0.3ex\hbox{$\triangleleft$}}}
\declaretheorem[style=myplain,numberwithin=section]{theorem}
\declaretheorem[style=myplain,numberlike=theorem]{lemma}

\declaretheorem[style=mydefinition,numberlike=theorem,qed=\envqed]{definition}
\declaretheorem[style=mydefinition,numberlike=theorem,qed=\envqed]{remark}

\usepackage[plainpages=false,pdfpagelabels,hidelinks,unicode]{hyperref}

\usepackage{color}
\usepackage{graphicx}
\usepackage[small]{caption}
\usepackage{subcaption}

\begingroup\expandafter\expandafter\expandafter\endgroup
\expandafter\ifx\csname pdfsuppresswarningpagegroup\endcsname\relax
\else
\fi

\usepackage{booktabs}
\usepackage{rotating}
\usepackage{multirow}

\usepackage{enumitem}

\usepackage{ifluatex}
\ifluatex
  \usepackage[no-math]{fontspec}
\else
  \usepackage[T1]{fontenc}
\fi
\usepackage{newpxtext,newpxmath}
\usepackage{comment}

\input{macros}

\newcommand{\orcid}[1]{ORCID:~\href{https://orcid.org/#1}{#1}}
\usepackage{authblk}

\newenvironment{keywords}{\par\textbf{Key words.}}{\par}

\title{The domain-of-dependence stabilization for cut-cell meshes is fully discretely stable}

\author[1]{Louis~Petri\thanks{\orcid{0009-0005-5322-3982}}}
\author[2]{Gunnar~Birke\thanks{\orcid{0000-0002-2008-6679}}}
\author[2]{Christian~Engwer\thanks{\orcid{0000-0002-6041-8228}}}
\author[1]{Hendrik~Ranocha\thanks{\orcid{0000-0002-3456-2277}}}

\affil[1]{Institute of Mathematics, Johannes Gutenberg University Mainz, Staudingerweg 9, 55128 Mainz, Germany}
\affil[2]{Applied Mathematics, University of M\"unster, Orl\'eans-Ring 10, 48149 M\"unster, Germany}

\date{May 6, 2026} 

\makeatletter
\hypersetup{pdfauthor={Louis Petri, Gunnar Birke, Christian Engwer, Hendrik Ranocha}} 
\hypersetup{pdftitle={The domain-of-dependence stabilization for cut-cell meshes is fully discretely stable}} 
\makeatother

\begin{document}

\maketitle

\begin{abstract}
\noindent
  \input{abstract.tex}
\end{abstract}

\begin{keywords}
  cut-cell meshes,
  discontinuous Galerkin methods,
  domain-of-dependence stabilization,
  semibounded operators,
  energy stability,
  Runge-Kutta methods
\end{keywords}


\section{Introduction}
\input{01_introduction}

\section{Preliminaries}
\input{01-1preliminaries}
\label{sec:preliminaries}

\section{Interpolation and quadrature properties}
\input{02_interp_quad_prop}

\label{chap:interp_quad_prop}

\section{Preparation of the proof}
\input{03_proof_prep}
\label{chap:proof_prep}

\section{Proof of the main result}
\input{04_stability_proof}
\label{chap:stability_proof}

\section{A pragmatic correction approach}
\label{sec:05_pragmatic_corr}
\input{05_pragmatic_corr.tex}

\section{Numerical tests}
\label{sec:06_numerics}
\input{06-1_numerics1D}

\section{Summary and conclusions}
\label{sec:07_summary}
\input{07_summary}

\section*{Acknowledgments}

\input{funding}

\appendix
\section{Postponed proofs}
\label{sec:appendix}
\input{08_appendix}

\printbibliography

\end{document}

%% file: macros.tex
\usepackage{xparse}

\let\epsilon\varepsilon
\let\phi\varphi
\let\rho\varrho

\usepackage{xspace}

\renewcommand{\i}{\mathrm{i}}

\newcommand{\hb}{\hat{B}}
\newcommand{\hl}{\hat{L}}
\newcommand{\hr}{\hat{R}}
\newcommand{\hx}{\hat{x}}
\newcommand{\bx}{\bar{x}}
\newcommand{\tx}{\widetilde{x}}
\newcommand{\tl}{\widetilde{l}}

\newcommand{\bom}{\bar{\omega}}
\renewcommand{\hom}{\hat{\omega}}
\newcommand{\htheta}{\hat{\theta}}
\newcommand{\Ip}{\mathbb{I}}

%% file: abstract.tex
We present a fully discrete stability analysis of the
domain-of-dependence stabilization for hyperbolic problems. The method
aims to address issues caused by small cut cells by redistributing
mass around the neighborhood of a small cut cell at a semi-discrete
level. Our analysis is conducted for the linear advection model
problem in one spatial dimension. We demonstrate that fully discrete
stability can be achieved under a time step restriction that does not
depend on the arbitrarily small cells, using an operator norm
estimate. Additionally, this analysis offers a detailed understanding
of the stability mechanism and highlights some challenges associated
with higher-order polynomials. We also propose a way to mitigate these
issues to derive a feasible CFL-like condition.
The analytical findings, as well as the proposed
solution are verified numerically in
one- and two-dimensional simulations.

%% file: 01_introduction.tex

Constructing a body-fitted mesh is time-consuming in cases where computational domains have irregular geometries.
This can be avoided by using an efficient cut-cell mesh that extracts the computational domain from a Cartesian background mesh.
This approach has become increasingly popular over the past decades,
most recently for hyperbolic partial differential equations (PDEs).
When employing discontinuous Galerkin (DG) methods in space, a cut-cell mesh poses several challenges:
besides the arbitrarily shaped cells, the use of explicit solvers for time integration results in a CFL condition that requires a very small time step, which affects both stability and the correct domain of dependence.

In recent years, initial progress has been made in the field of hyperbolic cut-cell DG methods.
In \cite{Sticko2019,Guerkan2020,Fu2021}, an adaptation of the ghost penalty method for hyperbolic problems was examined, originating from elliptic cut-cell techniques.
Another approach is the state redistribution method
\cite{Giuliani2021,TAYLOR2025}, which was originally developed for
finite volume schemes and has recently been extended to DG methods \cite{BERGER2021}.
In this work, we analyze the domain-of-dependence (DoD) stabilization introduced in \cite{engwer2020stabilized}.
The core idea of DoD is to add penalty terms into the semidiscretization that redistribute fluxes, adjusting the mass distribution within and between cells while restoring the domain of dependence for the downwind cell.
Promising results have been observed for nonlinear hyperbolic equations in one spatial dimension \cite{may2022dod} and for linear systems in two dimensions \cite{birke2023dod,birke2025error}.
Specifically, it has been shown in \cite{birke2023dod, streitbuerger2022} that the semidiscretization of DoD is stable for linear systems, and according to \cite{may2022dod}, its first-order case produces a monotone scheme.
A list of all proven DoD properties up to this point is given in Table~\ref{table:dod_porperties}.
Similar findings have been achieved for the state redistribution method \cite{TAYLOR2025,Berger2024}.
\begin{table}

\begin{center}
  \caption{Summary of the proven properties of DoD-stabilized DG schemes. Depending on the case considered, the details of the DoD methods may differ from the version introduced in this article.}
    \label{table:dod_porperties}
    \begin{tabular}{rll||r|r|r}
      &&& \multicolumn{1}{|c}{advection}
      & \multicolumn{1}{|c}{linear systems}
      & \multicolumn{1}{|c}{scalar nonlinear eq.}\\\hline\hline
      monotonicity, &fully discrete, &$P^0$
      &
        1D~\cite{engwer2020stabilized}&&1D~\cite{may2022dod} \\
      TVD, &fully discrete, &$P^0$
      &
        1D~\cite{engwer2020stabilized}&& \\
      TVDM, &fully discrete, &$P^1$
      &
        1D~\cite{engwer2020stabilized}&& \\
      SBP, &semidiscrete, &$P^k$
      &
        1D~\cite{petri2026kinetic} & &\\\hline
      $L^1$-stability, &fully discrete, &$P^0$ &1D~\cite{engwer2020stabilized}& & \\
      $L^2$-stability, &semidiscrete, &$P^k$
      & 2D~\cite{streitbuerger2022}
      & 1D~\cite{petri2026kinetic}, 2D~\cite{birke2026energy}
                    & 1D~\cite{may2022dod} \\
      \hline
      convergence, &fully discrete, &$P^0$ & 2D~\cite{birke2025error} & &
    \end{tabular}
    \end{center}
\end{table}

The novelty of this paper is, that
we extend the semidiscrete stability analysis of the DoD method to a fully discrete setting.
Specifically, we use an explicit Runge-Kutta method and examine whether the resulting fully discrete scheme is stable, meaning whether the numerical solution satisfies $\|u^{n+1}\| \le \|u^n\|$, $n \ge 0$ under a CFL-like time step restriction that is independent of the arbitrarily small cut cells.

Our analysis focuses on the one-dimensional case to avoid irregularly shaped cells and specifically examine the stabilization mechanism for the linear advection equation.
Building on studies that explore the strong stability of explicit Runge-Kutta (RK) methods
\cite{tadmor2002semidiscrete,sun2017stability,ranocha2018L2stability,sun2019strong, achleitner2024necessary},
we utilize the semidiscrete stability of the DoD stabilization to choose high-order time-integration schemes and reduce the fully-discrete stability issue to a CFL-like condition $\Delta t \le c \norm{L}_M \le C$, where $\norm{L}_M$ is the discrete operator norm of the right-hand side $L$ of the semidiscrete system $u'(t) = L u(t)$.
For the unstabilized scheme, this scales inversely with the smallest cell size.
We show that applying DoD stabilization restores this scaling to that of the background cell by analyzing the involved mechanisms.
To do so, we will derive the corresponding system of ordinary differential equations (ODEs) for the semidiscretization and estimate the operator norms of the submatrices in the sparse system.
Using this approach, we can identify, explain, and improve additional mechanisms that impose further restrictions on the time step needed for stability.

This article is organized as follows.
We start by introducing the problem setting, the method, and elaborating on our matrix notation of the semidiscrete system in Section~\ref{sec:preliminaries}.
In Section~\ref{chap:interp_quad_prop}, we provide necessary estimates for the involved quadrature nodes and weights, which we use in Section~\ref{chap:proof_prep} for some initial partial operator norm estimates.
In Section~\ref{chap:stability_proof}, we prove the main result, which states strong stability for an appropriate time step that is independent of arbitrarily small cells.
An approach for improving other related problems is given in Section~\ref{sec:05_pragmatic_corr}.
Finally, we present some numerical results in 1D and 2D that confirm the theoretical observations.

%% file: 01-1preliminaries.tex
\subsection{Problem setting}\label{subsec:preliminaries_problem_setting}

For the scope of this work, we consider the linear advection equation
as a hyperbolic model problem in one spatial dimension
\begin{equation*}
\partial_t u(t,x) + a\partial_x u(t,x) = 0, \quad u(0,x)=u_0(x),
\end{equation*}
on a domain $\Omega=(x_L,x_R)$ with periodic boundary conditions
and $a>0$. We focus on this simple hyperbolic model problem
to investigate the effects of the DoD stabilization alone.
To discretize the spatial domain, we split it into $N$ cells
$E_i = (x_{i-1}, x_i)$ for $i = 1, \hdots, N$ with vertices
$x_0 < x_1 < x_2 < \hdots < x_{N-1} < x_N$. To simulate cut cells,
we include a cut by the vertex $x_c$, such that we have the full set
of vertices given by
$x_0 < x_1 < x_2 < \hdots < x_{c-1} < x_c < x_{c+1} < \hdots < x_{N-1} < x_N$
with the cell sizes
\begin{align*}
\abs{E_i} &= \Delta x_i = \Delta x, \quad i \in \{1,\hdots, N\} \setminus \{c, c+1\}, \\
\abs{E_c} &= \Delta x_c = \alpha \Delta x, \quad
\abs{E_{(c+1)}} = \Delta x_{c+1} = (1-\alpha) \Delta x.
\end{align*}
$\Delta x$ will be referred to as the background cell size.
We assume that the $\alpha \in [0, 1/2]$, which we call the
\textit{cut-cell factor}, such that $E_c$ may be an arbitrarily small cut cell,
while $E_{(c+1)}$ is also a cut cell that is limited by the size $ \Delta x/2$
from below. This spatial setting is visualized in Figure~\ref{fig:cut_cell_setting}.

\begin{figure}[htbp]
    \centering
   \begin{tikzpicture}
        \draw[gray, thick] (0,0) -- (12.5,0);
        \draw[gray, thick, dotted]  (0,0) -- (-0.5, 0);
        \filldraw[black] (1.25,-0.25)  node[anchor=north]{$\Delta x$};
        \filldraw[black] (1.25,0.0)  node[anchor=south]{$E_{(c-2)}$};

        \draw[gray, thick]  (2.5,1/4) -- (2.5, -1/4);
        \draw[gray, thick, dotted]  (2.5,0) -- (2.5, -4/5);
        \filldraw[black] (2.5,-4/5)  node[anchor=north]{$x_{c-2}$};
        \filldraw[black] (3.75,-0.25)  node[anchor=north]{$\Delta x$};
        \filldraw[black] (3.75,0.0)  node[anchor=south]{$E_{(c-1)}$};

        \draw[gray, thick]  (5,1/4) -- (5, -1/4);
        \draw[gray, thick, dotted]  (5,0) -- (5, -4/5);
        \filldraw[black] (5,-4/5)  node[anchor=north]{$x_{c-1}$};
        \filldraw[black ] (5.5,-0.25)  node[anchor=north]{$\alpha \Delta x$};
        \filldraw[black] (5.5,0.0)  node[anchor=south]{$E_{(c)}$};

        \draw[gray, thick]  (6,1/4) -- (6, -1/4);
        \draw[gray, thick, dotted]  (6,0) -- (6, -4/5);
        \filldraw[black] (6,-4/5)  node[anchor=north]{$x_{c}$};
        \filldraw[black] (6.75,-0.25)  node[anchor=north]{$(1-\alpha)\Delta x$};
        \filldraw[black] (6.65,0.0)  node[anchor=south]{$E_{(c+1)}$};

        \draw[gray, thick]  (7.5,1/4) -- (7.5, -1/4);
        \draw[gray, thick, dotted]  (7.5,0) -- (7.5, -4/5);
        \filldraw[black] (7.5,-4/5)  node[anchor=north]{$x_{c+1}$};
        \filldraw[black] (8.75,-0.25)  node[anchor=north]{$\Delta x$};
        \filldraw[black] (8.75,0.00)  node[anchor=south]{$E_{(c+2)}$};

        \draw[gray, thick]  (10,1/4) -- (10, -1/4);
        \draw[gray, thick, dotted]  (10,0) -- (10, -4/5);
        \filldraw[black] (10,-4/5)  node[anchor=north]{$x_{c+2}$};
        \filldraw[black] (11.25,-0.25)  node[anchor=north]{$\Delta x$};
        \filldraw[black] (11.25,0.00)  node[anchor=south]{$E_{(c+3)}$};
        \draw[gray, thick, dotted]  (12.5,0) -- (13, 0);
        \end{tikzpicture}
        \caption{One-dimensional cut-cell setting.}
        \label{fig:cut_cell_setting}
\end{figure}
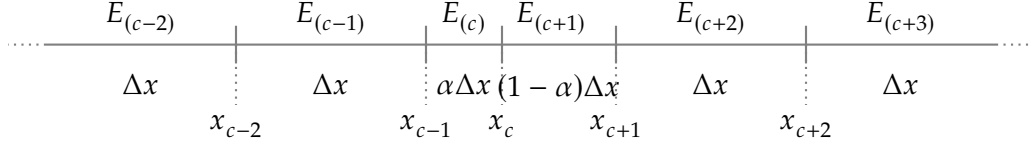

\subsection{Unstabilized scheme}
We follow the method of lines approach and discretize in space using a DG method.
Thus, we consider the polynomial space of degree $p$, denoted by $\mathcal{P}^p$, and the discrete function space
$$ \mathcal{V}_h^p(\Omega) = \left\{v_h \in L^2(\Omega) :v_h|_{E_i} \in \mathcal{P}^p(\Omega), \; i = 1,\hdots ,N \right\}.$$
As local basis on each cell we use a Lagrange basis,
either Gauss-Legendre (GL) or Gauss-Lobatto-Legendre (GLL).
Since the elements of $\mathcal{V}_h^p(\Omega)$ are not well-defined at the vertices,
we also introduce the jump $$\llbracket v_h \rrbracket_i = v_h(x_i)|_{E_{i}}-v_h(x_i)|_{E_{i+1}}.$$
Using a standard upwind flux \cite[Section 3.2.1, eq. (3.8)]{DiPietro2012book}, we obtain the semidiscrete scheme: find
$u_h(t) \in \mathcal{V}_h^p$ such that
\begin{equation}\label{eq:old:background_method}
\left( \partial_t u_h(t), v_h\right)_{L^2(\Omega)} + a_h\bigl( u_h(t), v_h \bigr) = 0, \quad \forall v_h \in \mathcal{V}_h^p,
\end{equation}
with the bilinear form
\begin{equation}\label{eq:old:ah_bilin}
    a_h(u_h, v_h)=-\sum\limits_{i=1}^Na\int\limits_{x_{i-1}}^{x_i}u_h\partial_x v_h \dif x + \sum\limits_{i=1}^Nau_h|_{E_i}(x_i) \llbracket v_h \rrbracket_i,
\end{equation}
where the first sum consists of all volume terms and the second sum of
all flux terms over the boundaries, resulting from integration by
parts.

We will refer to this as the background method, which can be applied
to both the equidistant background mesh without a cut cell and the
cut-cell mesh. For time stepping, we utilize an explicit Runge-Kutta method.

\subsection{DoD stabilization}
In the one-dimensional setting, this scheme has to deal with two problems
that arise because of an underlying cut-cell mesh
(\cite[Section~2.3]{streitbuerger2023diss}, \cite{engwer2020stabilized}):
\begin{itemize}
    \item[1.] Arbitrarily small cells
    \item[2.] Incorrect domain of dependence of the outflow cell
\end{itemize}
Facing these, we want to avoid the background CFL condition
$a\Delta t /(\alpha\Delta x) \le C$, as $\alpha$
may become arbitrarily small, leading to unfeasible time-step conditions.
To address this, we describe the DoD stabilization, applied as a
correction term on the semidiscretization as introduced in
\cite{engwer2020stabilized}.

\begin{definition}[Extension operator]
    Consider $u_h\in \mathcal{V}_h^p(\Omega)$. The extension operator $\mathcal{L}_E$ extends the polynomial function $u_h$ from cell $E$ to the whole domain $\Omega$ by
    \begin{equation}\label{eq:old:extension_operator}
        \mathcal{L}_E : \mathcal{V}_h^p(\Omega)|_{E}\rightarrow P^p(\Omega), \quad \text{where }\; \mathcal{L}_E(u_h) \in P^p(\Omega) \; \text{ with } \; \mathcal{L}_E(u_h)|_E=u_h|_E.
    \end{equation}
    Note that this is just the evaluation of the uniquely defined polynomial $u_h|_E$ outside $E$.
\end{definition}

The DoD stabilization will address the issues by treating the volume and flux terms
of \eqref{eq:old:ah_bilin} separately. To manage the mass flowing into cell
$E_c$ from $E_{c-1}$, we introduce the stabilization term
\begin{equation}\label{eq:old:J0}
J^{0,c}_h\left( u_h, v_h \right) = \eta_c a\left(\mathcal{L}_{E_{(c-1)}}(u_h)(x_c)-u_h|_{E_c}(x_c)\right)\llbracket v_h\rrbracket_c
\end{equation}
for an $\eta_c>0$, which will be specified soon.
We can consider $J^0$ as a redistribution of fluxes, as it decreases the
inflow into the small cell $E_c$ from the inflow cell $E_{c-1}$ and the
outflow from $E_c$ into the larger cut cell $E_{c+1}$.

The DoD stabilization uses an additional term
\begin{equation}\label{eq:old:J1}
    J^{1,c}_h\left( u_h, v_h \right) = \eta_c\int\limits_{x_{c-1}}^{x_c}a \left(\mathcal{L}_{E_{(c-1)}}(u_h)-u_h\right)\cdot \partial_x\left(\mathcal{L}_{E_{(c-1)}}(v_h)-v_h\right) \dif x.
\end{equation}
The term $J^1$ is intended to regulate the gradients within
the cell $E_c$ and is consequently represented by the volume term.
Note that we use the extension operator for the cell $E_{(c-1)}$, as this is
the only inflow cell of $E_c$ in our specific case. Since we are only considering
a single small cut cell, we omit the index $c$ in \eqref{eq:old:J0} and \eqref{eq:old:J1}
and combine them into the full DoD correction term $J_h = J_h^0 + J_h^1$.
The stabilized semidiscretization now takes its final form
\begin{equation}\label{eq:old:dod_semidiscretization}
    \left( \partial_t u_h(t), v_h\right)_{L^2(\Omega)}
    + a_h\left( u_h(t), v_h \right)
    + J_h\left( u_h(t), v_h \right) = 0, \quad \forall v_h \in \mathcal{V}_h^p.
\end{equation}

\begin{remark}\label{remark:multiple_cut_cells}
    In a setting with multiple cut cells that require stabilization,
    we need a correction term for each cut cell. As in the literature,
    we define a set $\mathcal{S}$ of cells
    in need of stabilization (in our 1D case, this set would be defined
    $\mathcal{S}=\{E_j : |E_j|\le \Delta x/2\}$) and then
    determining the stabilization terms by
    $$ J_h^0\left( u_h, v_h \right) = \sum\limits_{j \in \mathcal{S}}J_h^{0,j}\left( u_h, v_h \right), \quad  J_h^1\left( u_h, v_h \right) = \sum\limits_{j \in \mathcal{S}}J_h^{1,j}\left( u_h, v_h \right),$$
    as long as the cut cells are not adjacent. For this practically relevant case
    the stabilization needs to be extended by a weighted redistribution of similar DoD
    terms that gets applied to the whole neighborhood of the adjacent small cells. Here, 
    we follow the existing literature and omit this case, as it would exceed the already
    technical machinery, while not contributing to the main goal of this work.
\end{remark}

\subsection{Stability properties of DoD}
In \cite{streitbuerger2022}, it was shown that the semidiscretization \eqref{eq:old:dod_semidiscretization}
is stable in a semidiscrete sense, i.e., the semidiscrete solution $u_h$ satisfies
\begin{equation}\label{eq:semidiscrete_stability}
    \|u_h(t)\|_{L^2(\Omega)}\le\|u_h(0)\|_{L^2(\Omega)}
\end{equation}
for any $t\ge0$. In the following sections, we will prove that the DoD stabilization
is stable in a fully discrete sense. Specifically, by selecting an explicit RK method
to solve the semidiscrete system, there exists a time-step $\Delta t$ that is
independent of $\alpha$, such that
\begin{equation*}
    \forall n \in \mathbb{N}: \; \|u^{n+1}\|\le\|u^n\|.
\end{equation*}

\subsection{Choice of \texorpdfstring{$\eta_c$}{ηc}}
The parameter $\eta_c \in [0,1]$ determines ``how much'' we want to stabilize: The closer to $1$ it is, the more mass gets redistributed directly to $E_{c+1}$ from $E_{c-1}$.
There are multiple ways to motivate the choice of this parameter.

The simplest is to just consider the cell fraction $\alpha$: We stabilize by $\eta_c = 1-\alpha$, such that the mass inflow from $E_{c-1}$ into $E_c$ is exactly given by the ratio
of its relative volume $\alpha$.

A second choice would take the actual domain of dependence into account: By setting $\eta_c=1-\min(1,\alpha/\nu)$, where $\nu=a \Delta t/\Delta x$ is the Courant number,
we also take care on how many cells the information travels for a time-step. In the first-order case ($p=0$, time stepping with the explicit Euler method), the fully discrete update
for the outflow cell of the small cut cell will then receive the exact amount of information of the cells it depends on \cite[Section~3.1.1]{streitbuerger2023diss}.

The downside of the second choice is, that the resulting semi-discretization depends
on the time step size, which requires to compute the semi-discretization for every step, even for linear systems.
How to choose this parameter exactly remains an open question for now.
Nevertheless, since also the analysis in \cite{birke2025error} considers
a time-step independent $\eta_c$, this indicates that such a choice would be favorable.

For the sake of our analysis, we will refer to a general $\eta_c$ in the following by
\begin{equation}\label{eq:penalty_parameter_defi}
    \eta_c=1-\min\left\{1,\frac{\alpha}{\lambda_c}\right\},
\end{equation}
where $\lambda_c >0$. The previous cases can be restored by setting $\lambda_c = 1$ or changing $\lambda_c = \nu$.
In Section~\ref{sec:05_pragmatic_corr}, we will present a different choice of $\lambda_c$, based on our upcoming analysis.
\subsection{Achieving fully-discrete stability}
For the remaining work, we will formulate the DoD stabilization
in the specific version of the discontinuous Galerkin spectral element method (DGSEM)
and display the resulting semi-discrete system in matrix formulation.
We use this formulation to perform our fully discrete stability analysis.
More precisely, we will use collocation with Gauss-Legendre (GL) and Gauss-Lobatto-Legendre (GLL) nodes and quadrature. We want to emphasize
two important implications:
\begin{itemize}
\item Considering GL nodes for quadrature, the resulting scheme
      evaluates all integrals exactly for the linear problem, regardless of the
      choice of basis functions. Additionally, the solution is independent of a
      specific basis, as the polynomial on each cell is uniquely defined.
      The following analysis examines Lagrange basis functions at Gauss-Legendre
      nodes. However, since the operator norm does not depend on the chosen basis,
      the final result stated in Theorem~\ref{theorem:operatornorm_estimate_lin_adv}
      below still applies to the exact formulation \eqref{eq:old:dod_semidiscretization}
      using exact integrals.
\item Considering GLL nodes as quadrature and Lagrange basis nodes,
      we examine a discontinuous Galerkin spectral element method (DGSEM) with under-integration for constructing
      the mass matrix. Therefore, the resulting system of ODEs will be an approximation
      to the exact semidiscrete system \eqref{eq:old:dod_semidiscretization} for that
      specific choice. In this case, we will also rely on a semi-discrete stability
      argument that has not yet been explicitly formulated. However, following
      \cite{birke2023dod}, the integrals of $a_h$ involve polynomials of
      maximum degree $2p-1$, making the GLL quadrature exact \cite{kopriva2010quadrature}.
      Furthermore, the volume terms
      are solely handled by integration by parts, which is also maintained by the
      summation-by-parts property \cite{gassner2013skew} of the DGSEM.
\end{itemize}
Given a discrete norm $\|\cdot \|_M$ based on the specific DGSEM choice (details are provided in the next Section~\ref{subsec:intro_notation_norms}), we want to achieve the fully discrete stability, yielding
\begin{equation}\label{eq:fully_discrete_stability_in_discrete_norm}
    \forall n \in \mathbb{N}: \; \|u^{n+1}\|_M\le\|u^n\|_M.
\end{equation}
For that, we follow the techniques of \cite{ranocha2018L2stability,tadmor2002semidiscrete} to exploit the semidiscrete stability \eqref{eq:semidiscrete_stability} of the DoD stabilization:
Consider a spatial semidiscretization written in terms of the linear system $\partial_t  u=Lu$.
Given that this is stable, we know that the norm of the solution does not increase in time. We can derive
\begin{align*}
    0 \ge \frac{\text{d}}{\text{d}t}\norm{u}^2 = \left(\frac{\text{d}}{\text{d}t} u, u\right) + \left(u, \frac{\text{d}}{\text{d}t}u\right) = (Lu, u) + (u, Lu).
\end{align*}
This implies $\forall u\colon (Lu, u)\le0$, i.e., the spatial operator $L$ is semibounded.
Applying a suitable explicit Runge-Kutta method for time-stepping, we can exploit the semiboundedness of $L$ to obtain
\begin{equation}\label{eq:energy_estimate_ansatz}
    \| u^{n+1}\|_M^2 -  \| u^n\|_M^2
\end{equation}
non-positive under a condition $\Delta t \norm{L}_M\le C$, where $C$ just depends on the chosen Runge-Kutta method.
For example, we cite the values $C=1$ for SSPRK$(3,3)$ and $C= 0.67493$ for SSPRK$(10,4)$ from \cite{ranocha2018L2stability}.

In general, we obtain such a CFL-like condition for an explicit RK method
of order $p$ with $p$ stages, if $p \equiv 3 \;(\text{mod}\; 4)$
(while several more methods are leading to that), see \cite[Theorem 4.2]{sun2019strong}.
Moreover, Tadmor \cite{tadmor2024runge} proved stability in a weaker sense for a
broader class of RK methods. These results are not directly applicable to this framework
but hint that not only the strongly stable methods classified in
\cite{sun2019strong,achleitner2024necessary} ensure stability.
Therefore, we reach an arbitrarily high order and all boils down to determine the operator norm
$\norm{L}_M$, which usually scales with $1/\min\{\Delta x_i : i = 1,\hdots, N\}$.
This is the common CFL-like condition.
Thus, we will show that the stabilized scheme yields an operator norm that is
independent of arbitrarily small cells with size $\alpha \Delta x$.

We anticipate the answer to this question and state our main result by the following theorem, while we provide its final proof in Section~\ref{chap:stability_proof}.
\begin{theorem}\label{theorem:operatornorm_estimate_lin_adv}
    Consider the system of ODEs $\partial_t u = Lu$ that resembles the semidiscrete system \eqref{eq:old:dod_semidiscretization},
    treated exactly or by under-integration with GLL quadrature.
    There is a constant $C$, just depending on the polynomial degree $p$ and the chosen basis/nodes, such that
    \begin{equation*}
      \|L u\|_M \le \frac{Ca}{\Delta x} \|u\|_M,
    \end{equation*}
    i.e., the operator norm $\|L\|_M$ does not scale with the inverse $\alpha^{-1}$ of the small cut-cell size by applying DoD,
    as it would without that stabilization.
  \end{theorem}
  This implies fully-discrete stability under a cut-cell independent time step restriction, if we apply the DoD stabilization to a cut-cell grid.

\subsection{Introducing notation and norms}\label{subsec:intro_notation_norms}
Considering the semidiscrete system \eqref{eq:old:background_method} or \eqref{eq:old:dod_semidiscretization}, we have $N$ cells
and  due to the Lagrange basis functions of degree $p$, $u\in \mathbb{R}^{(p+1)N}$
will consist of the nodal values, and we can express it as
$$u = (u_1, u_2, \hdots, u_N)^T, \quad u_i\in \mathbb{R}^{(p+1)}, \; i\in \{1, \hdots, N\},$$
where $u_i=(u_{i0}, \hdots, u_{ip})^T\in \mathbb{R}^{p+1}$ are the values and
$l_i=(l_{i0}, \hdots, l_{ip})^T$ are the basis functions in cell $E_i$,
such that $u_h|_{E_i}= \sum_{j=0}^p u_{ij}l_{ij}$.
To simplify the notation later, we consider the reference element $R = [-1, 1]$ and an arbitrary cell $E_i = [x_{i-1}, x_i]$ with $\Delta x_i = x_i-x_{i-1}$.
We can map the solution from $R$ to $E_i$ by
$$ x(\gamma) = \frac{x_i+x_{i-1}}{2} + \gamma \frac{\Delta x_i}{2}.$$
Using the Lagrange basis $\tl = (\tl_0, \hdots, \tl_p)^T$ at reference GL/GLL points in $R$, we have
\begin{equation}\label{eq:reference_vs_physical_cell_representation}
     u|_{E_i} = \sum\limits_{k=0}^p u_{ik}l_{ik}(x(\gamma))=\sum\limits_{k=0}^p u_{ik}\tl_k(\gamma).
\end{equation}
Denote the GL/GLL nodes $\{\widetilde{x}_{0}, \hdots, \widetilde{x}_{Rp}\}$ and quadrature weights $\{\omega_0, \hdots ,\omega_p\}$ on the reference element $R$. Further, we consider the (GL or GLL) cut-cell nodes
$\{\xi_{i0}, \hdots, \xi_{ip}\}\in E_{i+1}=[x_i, x_i+\alpha \Delta x_i]$ and $\{\xi_{0}, \hdots, \xi_{p}\}$ as the transformed nodes to the \textit{reference cut cell} $[1,2\alpha]$.

Considering polynomials $u_h, v_h \in \mathcal{V}_h^p$ in their Lagrange basis
$u_h|_{E_i}= \sum_{j=0}^p u_{ij}l_{ij}, v_h|_{E_i}= \sum_{j=0}^p v_{ij}l_{ij}$,
collocation leads to the discrete inner product
\begin{align*}
(u_h|_{E_i}, v_h|_{E_i})_{M_i}&=\sum\limits_{k=0}^p\omega_k\Delta x_i/2
\sum_{j=0}^p u_{ij}\widetilde{l}_{ij}(\widetilde{x}_k)
\sum_{m=0}^p v_{im}\widetilde{l}_{im}(\widetilde{x}_k)\\
& = \sum\limits_{k=0}^p\omega_k\Delta x_i/2 u_{ik}v_{ik} = u_i^T M_i v_i
\end{align*}
for $i =1, \hdots ,N$.
Here, $\omega_k\Delta x_i/2$ are the scaled quadrature weights for the cell $E_i$
and the resulting mass matrix $M_i=\text{diag}(\omega_1, \hdots ,\omega_p)\Delta x_i/2$ is diagonal.
We can abbreviate this by writing $(u_i, v_i)_{M_i} := (u_h|_{E_i}, v_h|_{E_i})_{M_i}$.
From this point, we can define the discrete norms
\begin{align*}
\|u_i\|_{M_i} := \sqrt{(u_i, v_i)_{M_i}}, \quad \|u\|_M := \sqrt{\sum\limits_{i=1}^N\|u\|_{M_i}^2} = u^TMu,
\end{align*}
where $M = \text{diag}\left(M_1, M_2, \hdots, M_N\right)$ is the global mass matrix.
Analogue, we can define the norm at the reference element $R$ for $u \in \mathbb{R}^{p+1}$, that resembles the nodal values of a cell transformed to $R$:
$$\|u\|_{M_R} = u^TM_Ru, \quad M = \text{diag}(\omega_1, \hdots ,\omega_p).$$

Naturally we can define the operator norm for every introduced vector norm.
In the following, all the respective vector-induced operator norms are denoted by $\|A\|$, while its corresponding vector norm is uniquely defined by the
dimension of $A \in \mathbb{R}^{n\times n}$: For $n=(p+1)\times N$, the vector norm is $\|\cdot\|_{M}$, while for $n=(p+1)$ it is $\|\cdot\|_{M_R}$ or $\|\cdot\|_{M_i}, i \in 1,\hdots, N$.
We can unify the latter cases because they induce the same operator norm: Let $u\in\mathbb{R}^{p+1}, A\in \mathbb{R}^{(p+1)\times (p+1)}$
and use for the cause of clarification $\|A\|_i$ as the induced operator norm of $\|u\|_{M_i}$, then
\begin{equation}\label{eq:opnorm_equality}
    \|A\|_{M_i} = \max\limits_{u \in \mathbb{R}^{p+1}}\frac{\|Au\|_{M_i}}{\|u\|_{M_i}}
    = \max\limits_{u \in \mathbb{R}^{p+1}}\frac{\sqrt{\Delta x_i/2}\|Au\|_{M_R}}{\sqrt{\Delta x_i/2}\|u\|_{M_R}}
        = \max\limits_{u \in \mathbb{R}^{p+1}}\frac{\sqrt{\Delta x_j/2}\|Au\|_{M_j}}{\sqrt{\Delta x_j/2}\|u\|_{M_j}} = \|A\|_{M_j}.
\end{equation}
As usual, when we consider a polynomial space $\mathcal{V}_h^p$, the involved integrals
account for polynomials up to degree $2p$. Therefore, we need to be mindful of the
discrete inner product, which we use to approximate the $L^2$ inner product.
Since the Gauss-Legendre quadrature is exact for polynomials up to degree $2p+1$,
we have for this case
\begin{equation}
    (u,v)_{M} = (u,v)_{L^2([-1, 1])}, \quad u \in \mathcal{P}^p([-1, 1])
\end{equation}
In the case of Gauss-Lobatto-Legendre quadrature, we encounter the well-known issue
that it is exact for polynomials up to degree $2p-1$. Consequently, for arbitrary
$u$ and $v$ of degree $p$, we cannot guarantee exactness. However, with respect
to the discrete norm,
we know that the discrepancy to the exact value of the involved integral
is bounded by a factor $\sqrt{3}$, see \cite{Kopriva2017}, \cite[(5.3.2)]{Canuto2006book}.\\
\subsection{The matrix notation}\label{subsec:matrix_notation}
To prove Theorem~\ref{theorem:operatornorm_estimate_lin_adv}, we
reformulate the semidiscrete system \eqref{eq:old:dod_semidiscretization} as a system of ODEs
\begin{equation*}
    \partial_t u = Lu.
\end{equation*}
    The exact shape of $L$ will be derived in this section.
    We construct all matrices on the reference element, so we can explicitly
    highlight the dependence of the cell size. Therefore, we require some relations at first.
    \begin{definition}\label{defi:semidiscr_matrices}
    We define (see e.g. \cite[eqs.\ (3.31), (3.47)]{Kopriva2009})
    \begin{align*}
        D &= \left(\tl_j'(\widetilde{x}_{Rl})\right)_{l,j \in \{0, \hdots ,p\}},  &&D_{E_i}=\left(l_{ij}'(\xi_{il})\right)_{l,j \in \{0, \hdots ,p\}}, \quad&&\text{(the derivative matrix)}, \\
        \hr &= (\tl_{0}(1), \hdots, \tl_{p}(1)), &&\hr_{E_i} = (l_{i0}(x_i), \hdots, l_{ip}(x_i)), \quad&&\text{(the right boundary extrapolation matrix)}, \\
        \hl &= (\tl_{0}(-1), \hdots, \tl_{p}(-1)), &&\hl_{E_i} = (l_{i0}(x_{i-1}), \hdots, l_{ip}(x_{i-1})),  \quad&&\text{(the left boundary extrapolation matrix)}, \\
        \Ip &= \left(\tl(\xi_{l})\right)_{l,j \in \{0, \hdots ,p\}},  && \Ip_{E_i} = \left(l_{ji}(\xi_{li})\right)_{l,j \in \{0, \hdots ,p\}},    \quad&&\text{(the extrapolation matrix)}.
    \end{align*}
\end{definition}
\begin{lemma}\label{lemma:matrices_scaling_to_reference_element}
    It holds
    \begin{enumerate}
        \item $ D_{E_i} = \frac{2}{\Delta x_i}D, \; M_i = \frac{\Delta x_i}{2}M_R,$
        \item $\hr_{E_i} = \hr, \; \hl_{E_i} = \hl, \; \Ip_{E_i} = \Ip.$
    \end{enumerate}
\end{lemma}
\begin{proof}
For fixed $l, j$, we have by transformation $l_{ik}(x(\gamma))=\tl_k(\gamma)$.
    Therefore, $$\tl'_j(\widetilde{x}_{l})= \left(l_{ij}(x_{il}(\widetilde{x}_{l}))\right)'= l'_{ij}(x_{il})\cdot x'_{il}(\widetilde{x}_{l}) = l'_{ij}(x_{il})\frac{\Delta x_i}{2}.$$
The remaining statements follow just by the identity of transformation $l_{ik}(x(\gamma))=\tl_k(\gamma)$ as in \eqref{eq:reference_vs_physical_cell_representation}.
\end{proof}
We proceed inserting all the basis functions $l_1, \hdots, l_N$ into the semidiscrete system to construct $L$.
For $i, k = 1, \hdots, N,\; j=1, \hdots, p$ with $k \neq i$, we have $(\partial_t u_h,l_{ij})_{M_k} = 0, \;(\partial_t u_h,l_{ij})_{M_i} = (\partial_t u_i,l_{ij})_{M_i}$ and $(u_h,\partial_x l_{ij})_{M_i} = (u_i,\partial_x l_{ij})_{M_R}$, because $l_{ij}|_{\Omega\setminus E_i}=0$.
We obtain for the respective terms of \eqref{eq:old:background_method}
\begin{align}\label{eq:matrixform_lhs_and_volume}
    (\partial_t u_i,l_j)_{M_i}&=\frac{\Delta x_i}{2}\omega_j\partial_t u_{ij}, \quad \Rightarrow (\partial_t u_i,l)_{M_i} = \frac{\Delta x_i}{2}M\partial_t u_i  \\
    \label{eq:matrixform_lhs_and_volume2}
    (u_i,\partial_x l_j)_{M_i} &= \sum\limits_{k=0}^pu_k \frac{2}{\Delta x_i}D_{kj}\frac{\Delta x_i}{2}\omega_k= \sum\limits_{k=0}^p D^T_{jk}\omega_ku_k  , \quad \Rightarrow (u_i,\partial_x l)_{M_i}= D^TMu_i.
    \end{align}
    The remaining boundary terms also require the nodal values $u_{i-1}$ of the inflow neighbor cell.
    As their values at $x_{i-1}$ and $x_{i}$ are generally (here: for GL nodes) not represented by the nodal values, we need extrapolation to the vertices.
    Using Definition \ref{defi:semidiscr_matrices}, the boundary terms of \eqref{eq:old:ah_bilin}, with $v_h = l_{ij}$ are given by
    \begin{equation}\label{eq:matrixform_boundary}
        \begin{aligned}
            &u_i(x_{i})l_{ij}(x_{i})- u_{i-1}(x_{i-1})l_{ij}(x_{i-1}) = l_{ij}(x_{i})\hr_{E_{i}} u_{i} -l_{ij}(x_{i-1})\hr_{E_{i-1}} u_{i-1}, \quad j\in \{0, \hdots, p\}\\
            &\Rightarrow u_i(x_{i})l_i(x_{i})- u_{i-1}(x_{i-1})l_i(x_{i-1}) = \hr_{E_{i}}^T \hr_{E_{i}} u_{i}- \hl_{E_{i}}^T \hr_{E_{i-1}} u_{i-1} = \hr^T \hr u_{i}- \hl^T \hr u_{i-1}
        \end{aligned}
    \end{equation}
    Finally combining  \eqref{eq:matrixform_lhs_and_volume}, \eqref{eq:matrixform_lhs_and_volume2} and \eqref{eq:matrixform_boundary} transforming to an arbitrary cell $E_i, \; i \in \{1, \hdots, N\}$, we obtain
\begin{equation*}
M\partial_t u_i = \frac{2a}{\Delta x}\left(D^TMu_i - \hr^T\hr u_i +\hl^T\hr u_{i-1}\right)
\end{equation*}
for the background method.

 The same procedure can be applied to the stabilization terms \eqref{eq:old:J0} and \eqref{eq:old:J1}.
Using the notation introduced in \ref{subsec:preliminaries_problem_setting}, we have
\begin{align*}
    J_h^{0,c}(u_h, l_c)&= a\eta_c l_c(x_c)\left(\mathcal{L}_{E_{(c-1)}}(u_h)(x_c)-u_h|_{E_c}(x_c)\right), \\
    J_h^{0,c}(u_h, l_{(c+1)})&= -a\eta_c l_{(c+1)}(x_c)\left(\mathcal{L}_{E_{(c-1)}}(u_h)(x_c)-u_h|_{E_c}(x_c)\right), \\
    J_h^{0,c}(u_h, l_i)&= \mathbf{0}, \quad i\notin \{c, (c+1)\}, \quad \text{with} \; \mathbf{0} = (0,\hdots, 0)^T \in \mathbb{R}^{(p+1)}.
\end{align*}
Now, we want to specify all the appearing terms further. Using Lemma \ref{lemma:matrices_scaling_to_reference_element}, we can write $l_c(x_c)=(l_{c0}(x_c),\hdots, l_{cp}(x_c),)^T=\hr^T$ and $l_{c+1}(x_c)=\hl^T$.
With the same arguments, the extrapolation $B_{J^0}$ from $L_{(c-1)}$ to the outflow boundary $x_c$, that is given by $B_{J^0} u_{(c-1)} = \mathcal{L}_{E_{(c-1)}}(u_h)(x_c)$, we observe that $B_J$ is just the
extrapolation from the basis functions $l_{(c-1)}$ of the inflow cell to the outflow boundary of the cut cell $x_{c}$, i.e., $B_J=\sum_{k=0}^pl_{k(c-1)}(x_c)$.
The last term is just the already treated extrapolation to the outflow boundary and therefore given by $u_h|_{E_c}(x_c)=\hr u_c$. So we can conclude the matrix-reformulation of \eqref{eq:old:J0} to
\begin{equation*}\label{matrixform_J0}
    J_h^{0,c}\left(u_h, \begin{pmatrix} l_c\\ l_{(c+1)}\end{pmatrix}\right)=
    a\eta_c\begin{pmatrix}
        \hr^TB_{J^0} & -\hr^T\hr  \\
        -\hl^TB_{J^0} & \hl^T\hr
    \end{pmatrix}
    \begin{pmatrix}
        u_{(c-1)} \\ u_c
    \end{pmatrix}, \quad J_h^{0,c}(u_h, l_i) = \mathbf{0}, \quad i\notin \{c, (c+1)\}.
\end{equation*}
We abbreviate $\hb_R = \hr^T\hr,\; \hb_L = \hl^T\hr,\; \hat{B}_{J^0_1}=\hr^TB_{J^0}$ and $\hat{B}_{J^0_2}=\hl^TB_{J^0}$ in the remaining work.

Continuing with $J_h^{1,c}$, we just have a few non-zero terms again because of the integral over $E_c$:
\begin{align*}
    J_h^{1,c}(u_h, l_{(c-1)})&= a\eta_c \int\limits_{x_{c-1}}^{x_c}\left(\mathcal{L}_{E_{(c-1)}}(u_h)-u_h\right)\cdot \partial_x\mathcal{L}_{E_{(c-1)}}\left(l_{(c-1)}\right) \dif x, \\
    J_h^{1,c}(u_h, l_c)&= -a\eta_c \int\limits_{x_{c-1}}^{x_c}\left(\mathcal{L}_{E_{(c-1)}}(u_h)-u_h\right)\cdot \partial_x l_c\dif x, \\
    J_h^{1,c}(u_h, l_i)&= \mathbf{0}, \quad i\notin \{(c-1), c\}, \quad \text{with} \; \mathbf{0} = (0,\hdots, 0)^T \in \mathbb{R}^{(p+1)}.
\end{align*}
In these equations, we have to evaluate $\mathcal{L}_{E_{(c-1)}}(u_h)$ inside the cell $E_i$, i.e., at its quadrature nodes. Therefore, we can write this again as the extrapolation matrix $\Ip=\Ip_{E_c}$, such that
$\Ip u_{(c-1)}=\mathcal{L}_{E_{(c-1)}}(u_h)|_{E_c}$. From this point, we can proceed as for \eqref{eq:matrixform_lhs_and_volume2}:
\begin{align*}
    J_h^{1,c}(u_h, l_{ci})&=-a\eta_c\sum\limits_{j=0}^p\left((\Ip u_{(c-1)})_j-u_{c_j}\right)\cdot \frac{2}{\alpha \Delta x}D_{ji}\frac{\alpha \Delta x}{2}\omega_j = -a\eta_c\left(\sum\limits_{j=0}^pD^T_{ij}\omega_j(\Ip u_{(c-1)}-u_{c})_j\right) \\
    \; \Rightarrow J_h^{1,c}&(u_h, l_{c}) = -a\eta_c D^T M (\Ip u_{(c-1)}-u_c) =  a\eta_c
    \begin{pmatrix}-D^TM \Ip & D^T M \end{pmatrix} \begin{pmatrix}u_{(c-1)} \\ u_c \end{pmatrix}.
\end{align*}
To rewrite the last remaining term $J_h^{1,c}(u_h, l_c)$, we first have to focus on $\partial x \mathcal{L}_{E_{in}}(l_{(c-1)j}(x))$ in detail.

For a fixed $j\in \{0,\hdots, p\}$, the (exact) interpolation polynomial of $l'_{(c-1)j}(x)$ at its own nodes $\{x_{(c-1)0}, x_{(c-1)0}, ... , x_{(c-1)p}\}$ is given by itself, i.e.,
\begin{equation*}
    \partial x \mathcal{L}_{E_{in}}(l_{(c-1)j}(x)) = \partial_x l'_{(c-1)j}(x)= l'_{(c-1)j}(x).
\end{equation*}
As we finally want to use the matrices introduced in Definition \ref{defi:semidiscr_matrices}, we interpolate $l_{(c-1)w}(x)$ at the respective nodes $\{x_{c1}, x_{c2}, \hdots, x_{cp}\}$
and basis functions $l_c$ inside the cut cell, i.e. $l_{(c-1)w}(x) = \sum_{k=0}^p l_{ck}(x)l_{(c-1)j}(x_ck)$ (which is again exact because of matching polynomial degrees).
By derivation and inserting a quadrature node of the cut cell, we obtain the desired matrix-notation
$$ \partial_x l_{(c-1)j}(x_{cl})=\sum\limits_{k=0}^p l_{ck}'(x_{cl})l_{(c-1)j}(x_{ck})=(D\cdot \Ip)_{lj}.$$
Now, we can combine these pieces to obtain by the full quadrature
\begin{align*}
    J_h^{1,c}(u_h, l_{(c-1)j})&=a\eta_c\sum\limits_{k=0}^p \left((\Ip u_{(c-1)})_k-u_{c_k}\right)\omega_k l_{ck}'(x_{cl})l_{(c-1)j}(x_{ck})\\
    &= a\eta_c\sum\limits_{k=0}^p (D\cdot \Ip)^T_{jk} \omega_k \left((\Ip u_{(c-1)})_k-u_{c_k}\right),\\
    J_h^{1,c}(u_h, l_{(c-1)})&=a\eta_cD\Ip M\left(\Ip u_{(c-1)}-u_c\right)=  a\eta_c
    \begin{pmatrix}\Ip^TD^TM \Ip & -\Ip^TD^T M \end{pmatrix} \begin{pmatrix}u_{(c-1)} \\ u_c \end{pmatrix}.
\end{align*}
Finally, we have formulated every term of the semidiscretization \eqref{eq:old:dod_semidiscretization} in the matrix notation.

Combining all terms, we obtain the full system
\begin{equation*}
  \partial_t u(t) = Lu(t), \quad L \in \mathbb{R}^{(p+1)N \times (p+1)N}.
\end{equation*}
In particular, $L$ is given by
\begin{equation}\label{eq:dod_semidiscretization_matrix}
L = \left(
\begin{array}{*{13}c}
  L_1 &  &   &   &   &   &   &   &   &   L_{1L}  \\
  L_{2L} & L_2 &  &   &   &   &   &   &   &     \\
    & L_{3L} & L_3 &  &   &   &   &   &   &     \\
    &   & \ddots & \ddots &  &   &   &   &   &      \\
    &   &   & L_{(c-1)L} & L_{(c-1)} & L_{(c-1)R} &   &   &   &      \\
    &   &   &   & L_{cL} & L_c &  &   &   &      \\
    &   &   &   & L_{(c+1)LL}  & L_{(c+1)L} & L_{(c+1)} &  &   &    \\
    &   &   &   &   &   & \ddots & \ddots &  &     \\
    &   &   &   &   &   &   & L_{(N-1)L} & L_{(N-1)} &     \\
    &   &   &   &   &   &   &   & L_{NL} & L_N
\end{array}
\right),
\end{equation}
where the block matrices of the background scheme are given by
\begin{align*}
  L_i & = S_{i}M^{-1}\left(D^TM-\hat{B}_R\right), &\quad i \notin \{(c-1), c\}, \\
  L_{iL} & = S_{i}M^{-1}\hat{B}_L, &\quad i \notin \{c, (c+1)\},
\end{align*}
and the blocks affected by the DoD stabilization take the form
\begin{equation}\label{eq:block_matrices_extended}
  \begin{aligned}
  &L_{(c-1)} = S_{{c-1}}M^{-1}\left(D^TM-\hat{B}_R-\eta_c\Ip^TD^TM\Ip\right), \\
  &L_{(c-1)R} = S_{c-1}M^{-1}\left(\eta_c\Ip^TD^TM\right), \\
  &L_{cL} = S_{c}M^{-1}\left(\eta_c D^T M \Ip + \hb_L - \eta_c \hb_{J^0_1}  \right), \\
  &L_{c} = S_{c}M^{-1}\left( D^TM-\eta_cD^TM -\hb_R + \eta_c \hb_R \right), \\
  &L_{(c+1)LL} = S_{c+1}M^{-1}\eta_c \hb_{J^0_2}, \\
  &L_{(c+1)L} = S_{c+1}M^{-1}\left( \hb_L - \eta_c \hb_L\right),
  \end{aligned}
\end{equation}
with the scaling coefficients
\begin{equation*}
  S_{i} = \frac{2a}{\Delta x_i}.
\end{equation*}
We will use this construction to prove the main result of this work in Section~\ref{chap:stability_proof}, stated by Theorem~\ref{theorem:operatornorm_estimate_lin_adv}.

%% file: 02_interp_quad_prop.tex
For the estimates of partial operator norms in Section \ref{chap:proof_prep},
we require certain properties regarding interpolation nodes and quadrature
weights that depend on the polynomial degree and the type of nodes. Specifically,
these include the minimal distance between two nodes and the largest ratio of
two quadrature weights. Since we cannot calculate the nodes and weights exactly
for Gauss-Legendre quadrature or for Gauss-Lobatto-Legendre, we will utilize
known estimates and adapt them in this section to suit our requirements.

Since we are focusing on the reference element $[-1,1]$ in this section 
and have to consider the details of the different node types, 
we specify our notation by denoting the Gauss-Legendre
nodes as $\hx_0, \hdots, \hx_p \in (-1,1)$ and the Gauss-Lobatto-Legendre
nodes as $\bx_0, \hdots, \bx_p\in [-1, 1]$.
Similarly, we denote the respective quadrature weights as
$\hom_0, \hdots, \hom_p$ and $\bom_0, \hdots, \bom_p$.
As this specification will not be
required in the subsequent sections, we will shift then back to the notation 
$\widetilde{x}_0, \hdots, \widetilde{x}_p$, that represents both cases again.

The following estimates are based on known estimates for the Gauss-Legendre and
Gauss-Lobatto-Legendre nodes and weights: For the half of the nodes inside
$(-1, 0)$, we can write
\begin{align}
\hx_{j-1} &= -\cos(\htheta_j),  &j &\in \left\{ 1,\hdots, \left\lfloor\frac{p+1}{2}\right\rfloor \right\} \label{eq:nodes_defi_trigGL},\\
\bx_j &= -\cos(\bar{\theta}_j),  &j &\in \left\{ 1,\hdots, \left\lfloor\frac{p-1}{2}\right\rfloor \right\}\label{eq:nodes_defi_trigGLL},
\end{align}
for some $\htheta, \bar{\theta}\in (0, \frac{\pi}{2})$.
Note that we have $\bx_0=-1$ for every $p$ and $\bx_{\frac{p}{2}}=\hx_{\frac{p}{2}}=0$, if $p$ is even.
Then, we have from \cite[eqs. (2.3.14) and (2.3.15)]{Canuto2006book} that
\begin{align}
    \left(j-\frac{1}{2}\right)\frac{\pi}{p+1}< &\htheta_j <j\frac{\pi}{p+2},  &j &\in \left\{ 1,\hdots, \left\lfloor\frac{p+1}{2}\right\rfloor \right\} \label{eq:nodes_estimate_trigGL},\\
    j\frac{\pi}{p}< &\bar{\theta}_j <\left(j+1 \right)\frac{\pi}{p+1},  &j &\in \left\{ 1,\hdots, \left\lfloor\frac{p-1}{2}\right\rfloor \right\}. \label{eq:nodes_estimate_trigGLL}
\end{align}
We receive the other half of the nodes by symmetry to the origin with the same properties.
Further, detailed information about these estimates can be found in \cite{szegő1975orthogonal} and \cite{Parter1999Lobatto}.
As the proofs of Lemmas~\ref{lemma:weightss_min_quotient}~and~\ref{lemma:nodes_min_distance} are rather technical, we
provide them in the Appendices~\ref{appendix_proof_inter_quad_prop1}~and~\ref{appendix_proof_inter_quad_prop2}.
\begin{lemma}[Minimal weight quotient estimate.]
    \label{lemma:weightss_min_quotient}
    There are constants $\hat{K}_1, \bar{K}_1$, such that for all $i,j \in \{0,\hdots, p\}$
    \begin{equation*}
        \frac{\hom_i}{\hom_j}\le \hat{K}_1p, \quad \frac{\bom_i}{\bom_j}\le \bar{K}_1p,
    \end{equation*}
    i.e., the quotient of any weights increases at worst linearly in $p$.
\end{lemma}

\begin{lemma}[Minimal node distance estimate.]    \label{lemma:nodes_min_distance}
    There are constants $\hat{K}_2>0$ and $\bar{K}_2>0$, such that for all $i,j \in \{0,\hdots, p\}$ with $i\neq j$,
    \begin{equation*}
        \abs{\hx_i-\hx_j} \ge \frac{\hat{K}_2}{p^3}, \quad \abs{\bx_i-\bx_j} \ge \frac{\bar{K}_2}{p^3}.
    \end{equation*}
\end{lemma}

\begin{remark}\label{rmk:note_to_node_scaling}
We note that the cubic dependency on $p$ for the GLL case yields closer nodes for higher orders.
Numerical tests indicate that a bound $\bx_i-\bx_j\ge {C}/{p^2}$ (as we technically achieved inside the proof of Lemma 3.2 for the Gauss-Legendre nodes) seems more valid.
Using other boundary estimates for $\theta_j$, as in \cite[(Lemma 2.4)]{Parter1999Lobatto} might be sufficient for that.
\end{remark}

%% file: 03_proof_prep.tex
In this section, we analyze the operator norm of the various forms of the
extension operator $\mathcal{L}_{E_{in}}$ that is given by $\Ip$ in the matrix notation. Because of Lemma~\ref{lemma:matrices_scaling_to_reference_element},
we can consider the reference element, where $\Ip$ interpolates the solution
from $[-1,1]$ to $[1,1 + 2\alpha]$.

\begin{lemma} \label{lemma:interpol_estimate}
  There is a constant $C$ depending on the chosen type of basis (but not on $p$),
  such that
  \begin{equation*}
    \|\Ip u\|_{M_R} \le Cp^{\frac{3}{2}}\left((2+2\lambda_c)p^3\right)^{p} \|u\|_{M_R},
  \end{equation*}
  i.e., the operator norm of $\Ip$ scales (at most) exponentially in the polynomial degree $p$.
\end{lemma}
\begin{proof}
  The interpolation matrix is defined by the standard Lagrange interpolation polynomials $\tl$ defined at $\tx_0, \hdots, \tx_p$ and evaluated at $\xi_0, \hdots, \xi_p$.
  Thus, for $u \in \mathbb{R}^{p+1}$, we have
  \begin{align*}
      \|\Ip u\|_{M_R}^2 &=
      \left\|\left(
          \sum\limits_{i=0}^{p}\tl_i(\xi_{1}) u_i , \;
          \sum\limits_{i=0}^{p}\tl_i(\xi_{2}) u_i , \;
          \hdots , \;
          \sum\limits_{i=0}^{p}\tl_i(\xi_{p}) u_i
      \right)^T\right\|_{M_R}^2 \\
      &=  \sum\limits_{j=0}^{p}\omega_j\left(\sum\limits_{i=0}^{p}\tl_i(\xi_{j}) u_i \right)^2 = \sum\limits_{j=0}^{p}\omega_j\left(\sum\limits_{i=0}^{p}\tl_i^2(\xi_{j}) u_i^2 +2\sum\limits_{\substack{i,k=0\\i < k}}^{p}\tl_i(\xi_{j}) u_i\tl_k(\xi_{j}) u_k\right)\\
      &\le \sum\limits_{j=0}^{p}\omega_j\left(\sum\limits_{i=0}^{p}\tl_i^2(\xi_{j}) u_i^2 +2\sum\limits_{\substack{i,k=0\\i < k}}^{p}\frac{\tl_i^2(\xi_{j}) u_i^2}{2} + \frac{\tl_k^2(\xi_{j}) u_k^2}{2}\right) = p\sum\limits_{j=0}^{p}\omega_j\sum\limits_{i=0}^{p}\tl_i(\xi_{j})^2 u_i^2
  \end{align*}
  by applying Young's inequality in the estimate.
  We have to change the coefficients of the weights in front of every $u_i$.
  Therefore, we use Lemma~\ref{lemma:weightss_min_quotient} to obtain $K_1p \ge \max\{\frac{\omega_i}{\omega_j}, \; i,j \in \{0, \hdots, p\}\}$ for both types of nodes, such that
  we can redistribute and unite all $u_i\in \mathbb{R}$ to
  \begin{align*}
    \|\Ip u\|_{M_R}^2 &\le K_1p^2\sum\limits_{j=0}^{p}\left(\sum\limits_{i=0}^{p}\tl_j(\xi_{i})^2\right) \omega_j u_j^2\\
    &\le K_1p^2\left( \max\limits_{j \in \{0, \hdots p\}} \sum\limits_{i=0}^{p}\tl_j(\xi_{i})^2 \right) \|u\|_{M_R}^2\\
    &\le K_1p^3 \left(\max\limits_{j \in \{0, \hdots p\}}\tl_j(\xi_{p})^2 \right) \|u\|_{M_R}^2 \\
    &= K_1p^3 \left( \max\limits_{j \in \{0, \hdots p\}} \prod\limits_{\substack{m=0 \\ m \neq j}}^{p} \frac{1+ \alpha + \alpha \tx_p-\tx_m}{\tx_j-\tx_m} \right)^2 \|u\|_{M_R}^2 \\
    &\le K_1p^3 \left( \prod\limits_{\substack{m=0 \\ m \neq j}}^{p} \frac{2+ 2\lambda_c}{\tx_j-\tx_m} \right)^2 \|u\|_{M_R}^2,
  \end{align*}
  if we stabilize just for $\alpha \le \lambda_c$.

  As shown in Lemma~\ref{lemma:nodes_min_distance}, there is a constant $K_2$, such that the distance $\tx_i-\tx_{i-1}$ of every node pair is larger than ${K_2}/{p^3}$ for every $i$.
  Therefore, we can estimate the product further to
  \begin{equation*}
    \|\Ip u\|_{M_R}^2 \le K_1p^3 (2+2\lambda_c)^{2p}K_2\left(p^3\right)^{2p} \|u\|_{M_R}^2.
    \qedhere
  \end{equation*}
\end{proof}

The following statement about the derivative of the interpolation is similar to
the previous one; however, we require and indeed obtain a proportionality to $\alpha$
for its operator norm. The remaining factor will again exhibit an exponential scaling
in the magnitude $p^p$, which also applies to problems with larger cut cells.
\begin{lemma}\label{lemma:deriv_interpol_estimate}
There exists a $\mathcal{C}(p)$, that may depend exponentially on $p$, such that
  \begin{equation*}
    \|D\Ip u\|_{M_R} \le \alpha \mathcal{C}(p)\|u\|_{M_R}.
  \end{equation*}
  This holds for all cut cell factors $\alpha$.
\end{lemma}
The proof can be found in Appendix~\ref{appendix_proof_prep2} and is mainly based on the properties
of the differentiation matrix $D$ and the exponential scaling as in
Lemma~\ref{lemma:interpol_estimate}.

To cover all cases that will arise, we need another similar statement about the extension operator that gets evaluated at the inflow boundary of the outflow cell of the cut cell.
The required arguments were partially previously used in Lemma~\ref{lemma:interpol_estimate}.
\begin{lemma}\label{lemma:interpol_estimate_outflowbound}
  There is a constant $C$, such that
  \begin{equation*}
    \|\hb_{J_2^0}u\|_{M_R}\le \left(Cp^{\frac{3}{2}}((2+2\lambda_c)p^3)^{p}\right)\|u\|_{M_R}.  \end{equation*}
\end{lemma}
The proof can be found in Appendix~\ref{appendix_proof_prep3} and uses a similar idea
as in Lemma~\ref{lemma:interpol_estimate}.

\begin{remark}\label{remark:unstable_exptrapolation_verification}
  Lemmas~\ref{lemma:interpol_estimate}, \ref{lemma:deriv_interpol_estimate},
  and \ref{lemma:interpol_estimate_outflowbound} indicate that the operator norms
  $\|\Ip\|$, $\|D\Ip\|$, and $\|\hb_{J_2^0}\|$ may significantly increase with
  larger cut cells and higher-order spatial discretizations. The unstable extrapolation
  is the underlying cause of this effect. Although these estimates may not be sharp,
  we provide a numerical verification of the scalings in Figure~\ref{fig:verification_of_exp_scaling}, depicted for $\|\Ip\|$.

  \begin{figure}[htbp]
    \centering
    \begin{minipage}[t]{0.49\textwidth}
      \centering
    \includegraphics[width=0.99\textwidth]{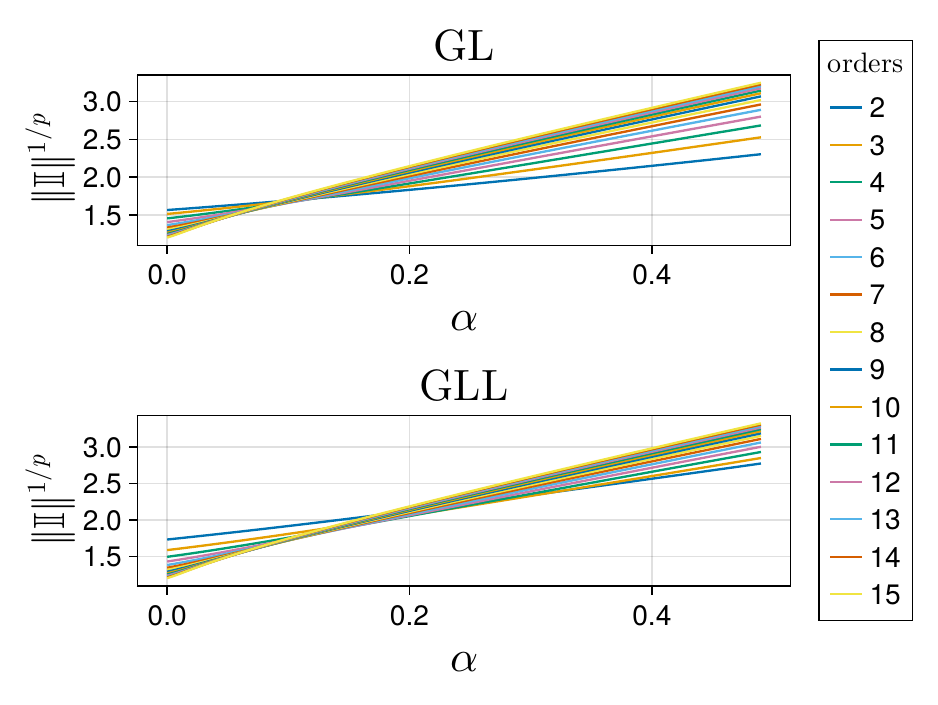}
    \end{minipage}
    \begin{minipage}[t]{0.49\textwidth}
      \centering
      \includegraphics[width=0.99\textwidth]{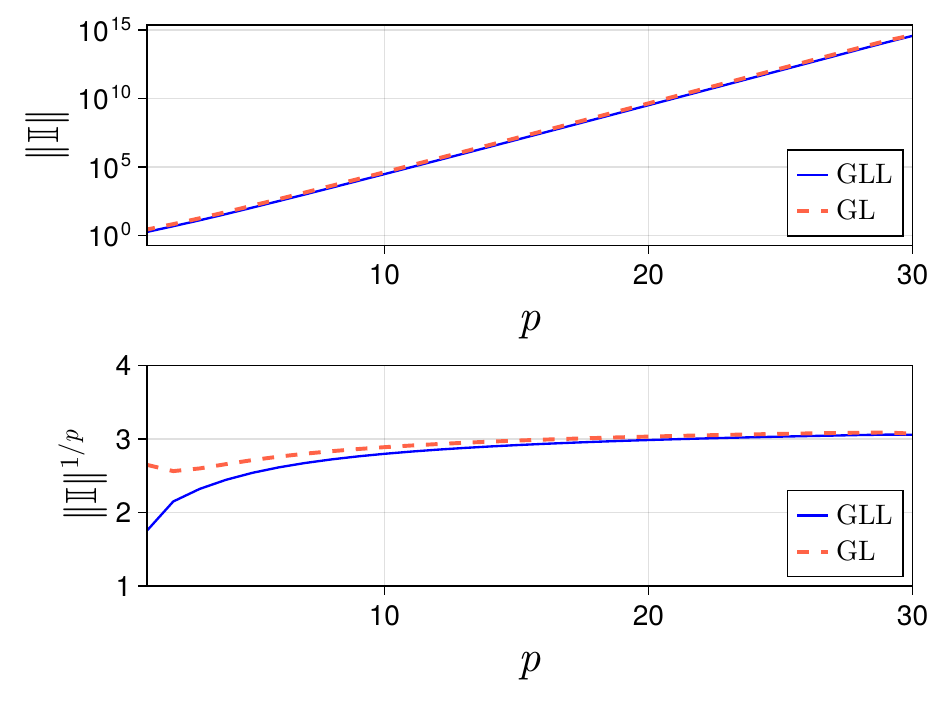}
    \end{minipage}
    \caption{Operator norm of $\|\Ip\|$ (top right) and $\|\Ip\|^{1/p}$ (else)
             from the reference element $[-1, 1]$ to the cut cell $[1,1+2\alpha]$.
             For the plots on the right, we use the fixed value $\alpha = 0.4$.}
    \label{fig:verification_of_exp_scaling}
  \end{figure}

  On the left side, we observe that all orders form a straight line,
  confirming that for a fixed $p$, the operator norm scales approximately
  with $\alpha^p$ (where we can assume $\alpha = \lambda_c$ by considering
  Lemma~\ref{lemma:interpol_estimate}). Similarly, we see in the top right
  that $\|\Ip\|$ is indeed exponential in $p$. However, taking the $p$-th root
  for every value results in a progression of $p^{\tau}$ for $\tau<1$.
  Thus, the scaling factor $(p)^{3p+\frac{3}{2}}$ in Lemma~\ref{lemma:interpol_estimate}
  is a slight overestimation. A noteworthy misleading estimate lies in
  $1+\alpha(1+\tx_p)-\tx_m\le2+2\lambda_c$, which leads to an amplifying factor,
  while its actual value may be less than $1$ and therefore primarily dampens
  the norm for small $\alpha$.

\end{remark}

\begin{lemma}\label{lemma:adjoint_norm}
  Consider the mass matrix $M_R\in\mathbb{R}^{(p+1)\times(p+1)}$ and an arbitrary $A \in \mathbb{R}^{(p+1)\times (p+1)}$.
  For the operator norm $\|\cdot\|$ induced by the vector norm given by $\|u\|^2_{M_R} = u^TM_Ru$,
  we have
  \begin{equation*}
    \|M_R^{-1}A^TM_R\| = \|A\|.
  \end{equation*}
\end{lemma}
\begin{proof}
  $\|A\|$ corresponds to the largest singular value of $A$. Therefore, we have $\|A^*\|=\|A\|$,
  where the adjoint $A^*$ is defined by the property $\langle v, Au \rangle_{M_R} = \langle A^*v, u\rangle_{M_R}$.
  Therefore, it remains to show that $A^* =  M_R^{-1}A^TM_R$.
  We conclude this by
  \begin{equation*}
    \langle u, M_R^{-1}A^TM_Rv\rangle_{M_R} = u^TM_RM_R^{-1}A^TM_Rv = u^TA^TM_Rv = \langle A^Tu, v \rangle_{M_R}.
    \qedhere
  \end{equation*}
\end{proof}

%% file: 04_stability_proof.tex
We have worked out all the tools required for Theorem~\ref{theorem:operatornorm_estimate_lin_adv} until now. In this section, we will provide proof and elaborate on the important terms occurring in the final estimate.
\begin{proof}[Proof of Theorem~\ref{theorem:operatornorm_estimate_lin_adv}]
We will follow a straightforward estimate by using our previously derived results. Consider $L$, as it is set up in \eqref{eq:dod_semidiscretization_matrix}.
For $u = (u_1, u_2, \hdots, u_N)^T\in \mathbb{R}^{(p+1)N}$, $u_n \in \mathbb{R}^{(p+1)}$ for $1\le n \le N$, we have
\begin{align*}
\|Lu\|_M^2 &=
\left\|\left(\begin{array}{c}
   L_1 u_1 + L_{1L} u_N \\
   L_{2L}u_1 + L_2u_2 \\
    \vdots\\
    L_{(c-1)L}u_{(c-2)}+L_{(c-1)}u_{(c-1)}+L_{(c-1)R}u_c\\
    L_{cL}u_{(c-1)}+L_cu_c \\
    L_{(c+1)LL}u_{(c-1)}+L_{(c+1)L}u_c+L_{(c+1)}u_{(c+1)}\\
    L_{(c+2)L}u_{(c+1)}+L_{(c+2)}u_{(c+2)}\\
    \vdots\\
    L_{(N-1)L}u_{(N-2)}+L_{(N-1)}u_{(N-1)}\\
    L_{NL}u_{N-1}+L_Nu_N\\
\end{array}\right)\right\|_M^2 \\
&= \|L_1 u_1 + L_{1L} uN\|_{M_1}^2 + \|L_{2L}u_1 + L_2u_2\|_{M_2}^2 + \hdots
\\ & \quad
+ \|L_{(c-1)L}u_{(c-2)}+L_{(c-1)}u_{(c-1)}+L_{(c-1)R}u_c\|_{M_{(c-1)}}^2  + \|L_{cL}u_{(c-1)}+L_cu_c\|_{M_c}^2
\\ & \quad + \|L_{(c+1)LL}u_{(c-1)}+L_{(c+1)L}u_c+L_{(c+1)}u_{(c+1)}\|_{M_{(c+1)}}^2  + \|L_{(c+2)L}u_{(c+1)}+L_{(c+2)}u_{(c+2)}\|_{M_{(c+2)}}^2
+ \hdots
\\ & \quad + \|L_{(N-1)L}u_{(N-2)}+L_{(N-1)}u_{(N-1)}\|_{M_{(n-1)}}^2 + \|L_{NL}u_{N-1}+L_Nu_N\|_{M_N}^2
,
\end{align*}
where $\|\cdot\|_{M_i}^2 = \Delta x_i \|\cdot\|_{M_R}^2$ yields the discrete norm in cell $i$.
By applying the triangle inequality to each vector that appears, we also obtain mixed terms, which we further separate into their natural components using Young's inequality, e.g.,
 \begin{align*}
  \|L_{nL}u_{n-1} + L_nu_n\|_{M_n}^2
  &\le
  \|L_{nL}u_{n-1}\|_{M_n}^2 + 2\|L_{nL}u_{n-1}\|_{M_n}\|L_nu_n\|_{M_n} + \|L_nu_n\|_{M_n}^2
  \\
  &\le
  2\left(\|L_{nL}u_{n-1}\|_{M_n}^2 + \|L_nu_n\|_{M_n}^2 \right).
 \end{align*}
Therefore, we can estimate all terms as
\begin{align*}
\|Lu\|_M^2 &\le 2\left(\|L_{1}u_{1}\|_{M_1}^2 + \|L_{1L}u_N\|_{M_1}^2 \right) + 2\left(\|L_{2L}u_{1}\|_{M_2}^2 + \|L_2u_2\|_{M_2}^2 \right) + \hdots \\
& \quad + 3\left(\|L_{(c-1)L}u_{c-2}\|_{M_{(c-1)}}^2 + \|L_{(c-1)}u_{(c-1)}\|_{M_{(c-1)}}^2 + \|L_{(c-1)R}u_c\|_{M_{(c-1)}}^2\right) + 2\left(\|L_{cL}u_{c-1}\|_{M_c}^2 + \|L_cu_c\|_{M_c}^2 \right)\\
& \quad + 3\left(\|L_{(c+1)LL}u_{(c-1)}\|_{M_{(c+1)}}^2 + \|L_{(c+1)L}u_{c}\|_{M_{(c+1)}}^2 + \|L_{(c+1)}u_{(c+1)}\|_{M_{(c+1)}}^2\right) \\
& \quad + 2\left(\|L_{(c+2)L}u_{(c+1)}\|_{M_{(c+2)}}^2 + \|L_{(c+2)}u_{(c+2)}\|_{M_{(c+2)}}^2 \right)
+ \hdots + 2\left(\|L_{NL}u_{N-1}\|_{M_N}^2 + \|L_Nu_N\|_{M_N}^2 \right)
\end{align*}
and then apply the compatibility of the operator norm and collect all terms to achieve
\begin{align*}
  \|Lu\|_M^2 &\le 2\left(\sum\limits_{n \notin \{(c-1), c, (c+1)\}}\left(\|L_n\|_{M_n}^2+\|L_{(n+1)L}\|_{M_n}^2\right)\|u_i\|_{M_n}^2\right) \\
  &\quad + 3\left(\|L_{(c-1)}\|_{M_{(c-1)}}\|u_{(c-1)}\|_{M_{(c-1)}}^2 + \|L_{(c-1)R}\|_{M_{(c-1)}}^2\|u_{c}\|_{M_{(c-1)}}^2 + \|L_{(c-1)L}\|_{M_{(c-1)}}^2\|u_{(c-2)}\|_{M_{(c-1)}}^2\right) \\
  &\quad + 2\left(\|L_{c}\|_{M_{c}}^2\|u_{c}\|_{M_{c}}^2 + \|L_{cL}\|_{M_{c}}^2\|u_{(c-1)}\|_{M_{c}}^2\right) \\
  &\quad + 3\left(\|L_{(c+1)LL}\|_{M_{(c+1)}}^2\|u_{(c-1)}\|_{M_{(c+1)}}^2 + \|L_{(c+1)L}\|_{M_{(c+1)}}^2\|u_{c}\|_{M_{(c+1)}}^2 + \|L_{(c+1)}\|_{M_{(c+1)}}^2\|u_{(c+1)}\|_{M_{(c+1)}}^2\right) \\
  &\quad + 2\|L_{(c+2)L}\|_{M_{(c+2)}}^2\|u_{(c+1)}\|_{M_{(c+1)}}^2 \\
  &= 2\left(\sum\limits_{n \notin \{(c-1), c, (c+1)\}}\left(\|L_n\|_{M_n}^2+\|L_{(n+1)L}\|_{M_n}^2\right)\|u_i\|_{M_n}^2\right) + \|L_{(c-1)L}\|_{M_{(c-1)}}^2\|u_{(c-2)}\|_{M_{(c-2)}}^2 \\
  &\quad + \left( 3\|L_{(c-1)}\|_{M_{(c-1)}}^2 +2\alpha\|L_{cL}\|_{M_{c}}^2 + 3(1-\alpha) \|L_{(c+1)LL}\|_{M_{(c+1)}}^2 \right)\|u_{(c-1)}\|_{M_{(c-1)}}^2 \\
  &\quad + \left(\frac{3}{\alpha}\|L_{(c-1)R}\|_{M_{(c-1)}}^2 + 2 \|L_{c}\|_{M_{c}}^2 + 3\frac{1-\alpha}{\alpha}\|L_{(c+1)L}\|_{M_{(c+1)}}^2 \right)\|u_{c}\|_{M_{c}}^2\\
  &\quad + \left(3\|L_{(c+1)}\|_{M_{(c+1)}}^2 + \frac{2}{1-\alpha}\|L_{(c+2)L}\|_{M_{(c+2)}}^2 \right)\|u_{(c+1)}\|_{M_{(c+1)}}^2 \\
  & \le \|u\|_M^2\max
  \begin{cases}
    2\|L_n\|_{M_n}^2+3\|L_{(n+1)L}\|_{M_n}^2,  &(n \notin \{(c-1), c, (c+1)\}) \\
    \|L_{(c-1)}\|_{M_{(c-1)}}^2 +2\alpha\|L_{cL}\|_{M_{c}}^2 + 3(1-\alpha) \|L_{(c+1)LL}\|_{M_{(c+1)}}^2, \\
    \frac{3}{\alpha}\|L_{(c-1)R}\|_{M_{(c-1)}}^2 + 2 \|L_{c}\|_{M_{c}}^2 + 3\frac{1-\alpha}{\alpha}\|L_{(c+1)L}\|_{M_{(c+1)}}^2, \\
    \|L_{(c+1)}\|_{M_{(c+1)}}^2 + \frac{2}{1-\alpha}\|L_{(c+2)L}\|_{M_{(c+2)}}^2.
  \end{cases}
\end{align*}
We now apply the specific terms of the block matrices, as shown in
\eqref{eq:block_matrices_extended}. To do this, we omit the indices of the operator,
as they all coincide, as noted in \eqref{eq:opnorm_equality}.

$2\|L_n\|^2+3\|L_{(n+1)L}\|^2$ originate (for $n \notin \{(c-1), c, (c+1)\}$) from the background method. Therefore, these terms will not be responsible for any problems and we estimate in that case
\begin{align*}
  2\|L_n\|_{M_n}^2+3\|L_{(n+1)L}\|^2 &\le 3\left(\|L_n\|^2+\|L_{(n+1)L}\|^2\right) \\
    &\le 3S_n^2\left(2\|M^{-1}\hb_R\|^2+2\|M^{-1}D^TM\|^2+\|M^{-1}\hb_L\|\right) \\
    &= 3S_n^2\left(2\|M^{-1}\hb_R\|^2+2\|D\|^2+\|M^{-1}\hb_L\|\right),
\end{align*}
again by using the Cauchy-Schwarz and Young's inequalities and applying Lemma \ref{lemma:adjoint_norm} afterwards, as we also will do in the following to split the operator norms.

Applying Lemma \ref{lemma:deriv_interpol_estimate} and the property $(1-\eta_c)/\alpha = \lambda_c^{-1}$, we continue with
\begin{align*}
  \quad
  \|L_{(c-1)}\|^2 +2\alpha\|&L_{cL}\|^2 + 3(1-\alpha) \|L_{(c+1)LL}\|^2
  \\
  &\le
  6\left( \|S_{c-1}M^{-1}D^TM\|^2+\|S_{c-1}M^{-1}\hb_R\|^2 + \|\eta_c S_{c-1}M^{-1}\Ip^TD^TM\Ip\|^2 \right) \\
  & \quad + 4\alpha\left(\|S_{c}\left(1-\eta_c\right)M^{-1}\hb_L\|^2 + \|S_{c}\eta_cM^{-1}MD\Ip\|^2 \right)
  +6(1-\alpha)\|\eta_c S_{c+1}M^{-1}\hb_{J_0^2}\| \\
  &\le S_{n}^2\left(6\left( \|D\|^2+\|M^{-1}\hb_R\|^2 + \eta_c^2\alpha^2 \mathcal{C}(p)^2 \|\Ip\|^2 \right) \right.\\
  & \quad \quad \quad + 4\alpha \left(\frac{1}{{\lambda_c}^2}\|M^{-1}\hb_L\|^2 + \eta_c^2\mathcal{C}(p)^2 \right) \left.+ 6\frac{\eta_c^2}{1-\alpha} \|M^{-1}\hb_{J_0^2}\|    \right).
\end{align*}
Next, we examine the terms originating from the cut cell itself:
\begin{align*}
 \quad
 \frac{3}{\alpha}\|&L_{(c-1)R}\|^2 + 2 \|L_{c}\|^2 + 3\frac{1-\alpha}{\alpha}\|L_{(c+1)L}\|^2
 \\
 &\le \frac{3}{\alpha}\|S_{c-1}\eta_cM^{-1}\Ip^TD^TM\|^2\\
 & \quad + 4\left(\|S_c(1-\eta_c)M^{-1}\hb_R\|^2+\|S_c(1-\eta_c)M^{-1}D^TM\|^2\right)
  + 3\frac{1-\alpha}{\alpha}\|S_{c+1}(1-\eta_c)M^{-1}\hb_L\|^2\\
 & \le S_n^2\left(3\eta_c^2\alpha\mathcal{C}(p)^2 + \frac{4}{\lambda_c^2}\left(\|M^{-1}\hb_R\|^2 + \|D\|^2\right) + 3\frac{\alpha}{\lambda_c^2(1-\alpha)}\|M^{-1}\hb_L\|^2\right).
\end{align*}
The final term we need to look at can be estimated similar to the terms of the background method, resulting in
\begin{align*}
  \|L_{(c+1)}\|^2 + \frac{2}{1-\alpha}\|L_{(c+2)L}\|^2 =S_n^2\left(\frac{6}{(1-\alpha)^2}\left(\|D\|^2+ \|M^{-1}\hb_R\|^2\right) + \frac{4}{1-\alpha}\|M^{-1}\hb_L\|^2\right).
\end{align*}
We can now combine these results to receive an estimate for the full operator norm:
 \begin{equation}\label{eq:operatornrom_estimate}
    \|Lu\|_M^2 \le \frac{4a^2\|u\|_M^2}{{\Delta x}^2}\max
      \begin{cases}
        3\left(2\|M^{-1}\hb_R\|^2+2\|D\|^2+\|M^{-1}\hb_L\|\right), \\
        6\left( \|D\|^2+\|M^{-1}\hb_R\|^2 + \eta_c^2\alpha^2 \mathcal{C}(p)^2 \|\Ip\|^2 \right) \\
        \quad + 4\alpha \left(\frac{1}{{\lambda_c}^2}\|M^{-1}\hb_L\|^2 + \eta_c^2\mathcal{C}(p)^2 \right) + 6\frac{\eta_c^2}{1-\alpha} \|M^{-1}\hb_{J_0^2}\|, \\
      3\eta_c^2\alpha\mathcal{C}(p)^2 + \frac{4}{\lambda_c^2}\left(\|M^{-1}\hb_R\|^2 + \|D\|^2\right) + 3\frac{\alpha}{\lambda_c^2(1-\alpha)}\|M^{-1}\hb_L\|^2,\\
      \frac{6}{(1-\alpha)^2}\left(\|D\|^2+ \|M^{-1}\hb_R\|^2\right) + \frac{4}{1-\alpha}\|M^{-1}\hb_L\|^2.
      \end{cases}
\end{equation}
None of the remaining terms depend on $\alpha^{-1}$: The remaining $\alpha$-dependencies can be estimated by $\alpha\le 0.5$. Besides the operator norms that do not originate from the background method,
the remaining ones were estimated in Lemmas~\ref{lemma:interpol_estimate}, \ref{lemma:deriv_interpol_estimate}, and \ref{lemma:interpol_estimate_outflowbound}, where we also use $\alpha \le 0.5$ to find a $C$ just depending on $p$ and the chosen nodes,
  such that
  \begin{equation*}
    \|L u\|_M \le \frac{Ca}{\Delta x} \|u\|_M.
    \qedhere
  \end{equation*}
\end{proof}

\begin{remark}[Multiple cut cells in need of stabilization]
\label{remark:multiple_cut_cells_2}
  Note that this model case effectively addresses a single small cut cell. Extending the analysis to multiple cut cells leads to certain consequences, depending on their distance:
  Multiple small cut cells with
  \begin{itemize}
    \item a minimum of zero cells in between: As elaborated in Remark~\ref{remark:multiple_cut_cells}, this formulation of DoD is not applicable to these type of problems yet,
          so we exclude this case here.
    \item a minimum of one cell in between: The outflow cell of the first small cell is also the inflow cell of the second small cell.
          In this case, this cell experiences also both of the effects coming from the DoD stabilization, being an inflow and outflow cell.
          In terms of the operator norm, the submatrices affecting $u_{(c_1+1)}=u_{(c_2-1)}$ are scaled by $(1-\alpha_1)^{-2}$
          but also consist of the additional terms for stabilizing $u_{c_2}$, which leads to an even bigger operator norm.
          Nevertheless, these effects are locally additive contributions. As the full operator norm is determined by the largest local contribution,
          the final operator norm will remain in the same order of magnitude.
          This can be verified though simulations by the need of a smaller time step size.
    \item two or more cells in between: In this case, the effects described above vanish, such that the stabilizations are independent.
          Therefore, this does not influence our estimate. In reality, one of these two stabilizations determines the sharp operator norm.
          Since the operator norm in dependence of the cut-cell factor alpha is not monotone, further determinations of the cell dominating the operator norm can not be made in the general case.
  \end{itemize}
\end{remark}

As implied by Theorem~\ref{theorem:operatornorm_estimate_lin_adv}, the DoD stabilization addresses the small cell problem for arbitrary small cut cells successfully.
However, we can extract some more crucial information out of the operator norm estimate. As stated in Lemmas \ref{lemma:interpol_estimate}, \ref{lemma:deriv_interpol_estimate}, and \ref{lemma:interpol_estimate_outflowbound} and numerically observed (and moderated in its impact) in Remark~\ref{remark:unstable_exptrapolation_verification},
the respective operator norms have an exponential scaling in $p$ and tend to increase for larger cut cells, roughly about $(1+\alpha)^p$. We observe in simulations, that this correlates to our CFL restrictions:
Small cut cells can be stabilized perfectly in the sense that we obtain a similar CFL condition as for Figure \ref{fig:bad_opnorms}, where we display this behavior for Gauss-Lobatto nodes.
$\eta_c=1$ displays the case where $1-\eta_c$ aligns with the cut-cell size. Because the other choice of $\eta_c = 1-\min\left\{1,\alpha\Delta x/(a\Delta t)\right\}$ varies by $\Delta t$ and the transport speed $a$,
we show some variations of $\lambda_c$.
We observe that smaller Courant numbers tend to increase the operator norm, while larger ones seem to decrease it. We will next investigate the latter phenomenon.

\begin{figure}
  \begin{minipage}[t]{0.32\textwidth}
    \centering
    \includegraphics[width=\textwidth]{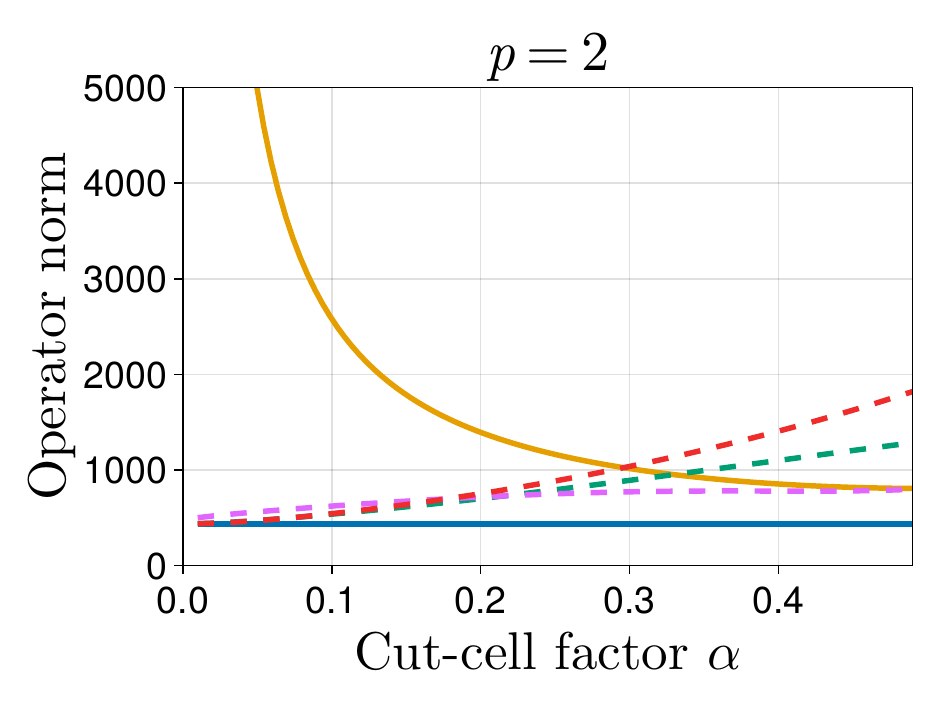}
  \end{minipage}
  \begin{minipage}[t]{0.32\textwidth}
    \centering
    \includegraphics[width=\textwidth]{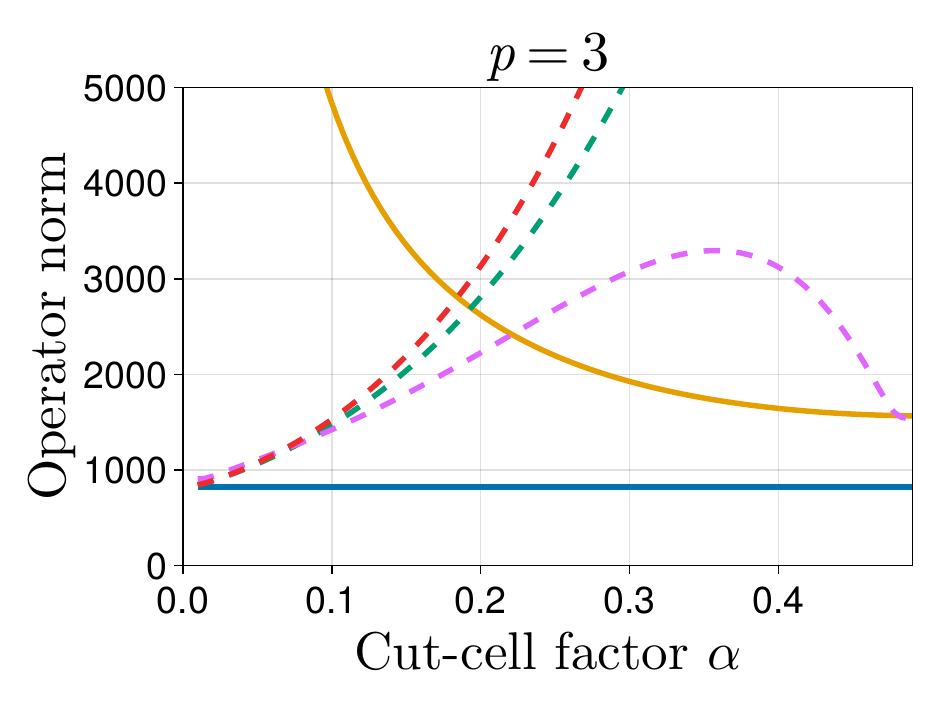}
  \end{minipage}
  \begin{minipage}[t]{0.32\textwidth}
    \centering
    \includegraphics[width=\textwidth]{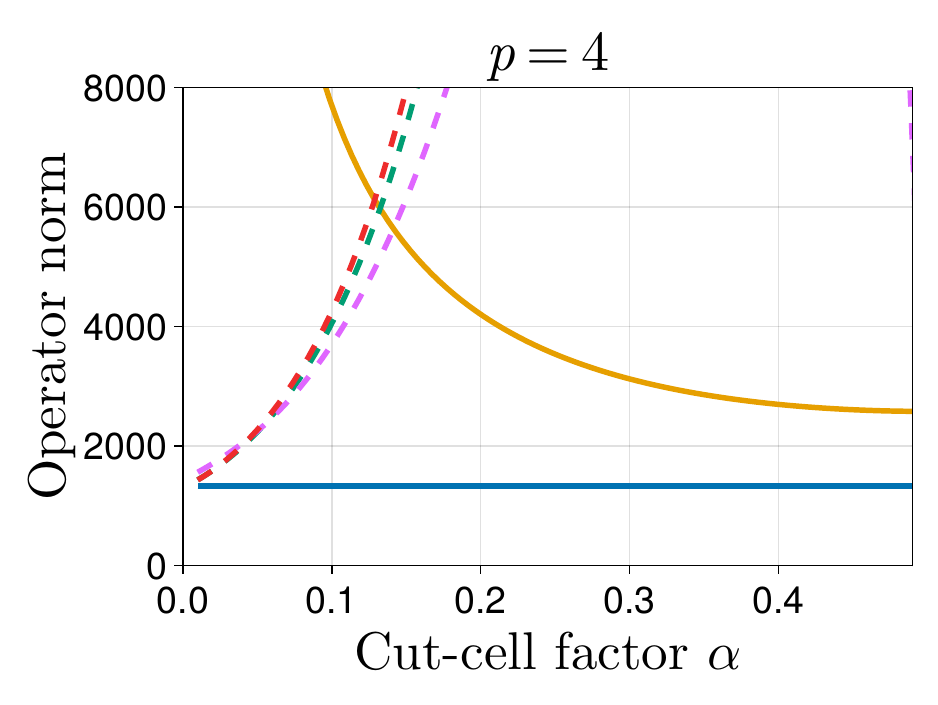}
  \end{minipage}
  \centering
  \includegraphics[width=0.6\textwidth, trim ={0 5cm 0 5cm} , clip]{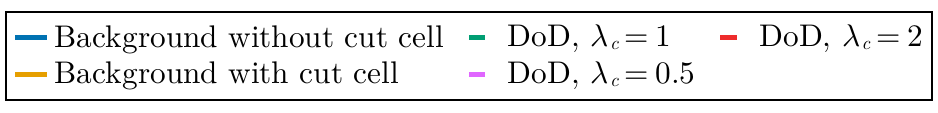}
  \caption{Operator norms for GLL nodes and different orders with $\Delta x = 1/50$.}
  \label{fig:bad_opnorms}
\end{figure}

\begin{remark}
  Other works such as \cite[eqs. (3.2a/b)]{Xu2019strong} that examine fully discrete stability use so-called inverse inequalities that estimate the volume and boundary terms and relate them to the mesh size. By using the matrix-type framework as in this article,
  this is equivalent to use the homogeneity of the operator norms to extract the scaling on the cell, e.g., $\|S_iDu\|=S_i\|Du\|\le S_i\|D\|\|u\|$ where $\|D\|$ does not depend on the spatial step size
  and the inverse dependency on  $\Delta x_i$ originates from $S_i$.
\end{remark}

%% file: 05_pragmatic_corr.tex
The undesired side effect of an increasing operator norm for larger
cut cells, when considering higher-order spatial discretizations in
combination with the DoD stabilization, limits its application to
cut-cell meshes. However, as shown in Figure \ref{fig:bad_opnorms} for
$p=3$, specific choices of $\lambda_c$ can mitigate the strong
increase of the operator norms.

Considering \eqref{eq:operatornrom_estimate}, there are particularly two mechanisms involved.
On one hand, the terms that depend on $\mathcal{C}(p)$ and $\Ip$ are responsible for the exponential increase in $p$ of $\|Lu\|_M$. Each of them has a coefficient of at least $\eta_c$.
On the other hand, we have the term $4\alpha/\lambda_c^2\|M^{-1}\hb_L\|^2$, which originates from $\|L_{cL}\|$ and describes the reduced inflow into the cut cell.
Smaller values for $\lambda_c$ lead to $\eta_c$ decreasing more rapidly according to \eqref{eq:penalty_parameter_defi} for increasing $\alpha$, reaching $\eta_c=0$ for $\alpha = \lambda_c$.
Thus, we can control the terms involving $\mathcal{C}(p)$ and $\Ip$ through the coefficient $\eta_c$.
However, reducing $\lambda_c$ obviously means decreasing the amount of stabilization.
Therefore, we pay a price elsewhere. Generally, this corresponds in \eqref{eq:operatornrom_estimate} to all terms with $\lambda_c$ in the denominator.
A numerical analysis of the operator norms of the involved block matrices $L_c, L_{cL}, \hdots$ suggests that the term $4\alpha/\lambda_c^2\|M^{-1}\hb_L\|^2$ dominates at some point.
Therefore, we need to find an appropriate $\lambda_c$, that stabilizes not too much and not too little, ensuring a balance between extrapolation and minor cell problems.
This motivates us to find optimized values for $\lambda_c$.
\subsection{Determination of the optimized $\lambda_c$}\label{subsec:05_determination}
First, we aim to provide explicit values for the choices of $\lambda_c$ that minimize the operator norm.
We calculate these by evaluating the operator norm of the right-hand side $L(\lambda_c, \alpha)$ for $\alpha \in [0, 0.5]$, $\lambda_c \in [0, 1]$, and extracting $\lambda_c$ such that
\begin{equation}\label{eq:minmax_opnorm}
    \min\limits_{\lambda_c\in \mathcal{S}_{\lambda_c}}\max\limits_{\alpha \in \mathcal{S}_\alpha}\{||L(\lambda_c, \alpha)||_M\}, \quad \mathcal{S}_{\lambda_c} = [0.001, 0.499], \; \mathcal{S}_{\alpha} = [0.01, 1],
\end{equation}
for a fixed order $p$, choice of nodes, and a grid of 51 cells.
We include the cut between cells 25 and 26 (the position did not significantly impact any experiments) and took 51 equidistant discrete values, including the boundaries of $\mathcal{S}_{\lambda_c}, \mathcal{S}_{\alpha}$, to evaluate expression \eqref{eq:minmax_opnorm}.
We present the resulting values for $\lambda_c$ in Table~\ref{table:lambda_c_values} and denote these as optimized $\lambda_c$ in the following. The origin of this effect, along with experiments, suggests that these values are independent of the number of cells and the present cut cells, assuming we do not need to stabilize multiple neighboring cut cells.

\begin{table}
\begin{center}
  \caption{Optimized values for $\lambda_c$ to obtain a minimal operator norm by \eqref{eq:minmax_opnorm}.}
    \begin{tabular}{p{0.6cm}|p{0.6cm}||p{2cm}|p{2cm}|p{2cm}|p{2cm}}
     \multicolumn{2}{c||}{Poly.Deg.} & \multicolumn{2}{|c|}{Gauss-Lobatto-Legendre} & \multicolumn{2}{|c}{Gauss-Legendre}\\
     \hline
     0& 6& 1.00000& 0.12877&1.00000 & 0.11699\\
     1& 7& 0.87665& 0.10218& 0.78913& 0.09643\\
     2& 8& 0.53986& 0.0862& 0.44159&0.07978 \\
     3& 9& 0.32132& 0.07213& 0.27871&0.06808 \\
     4& 10& 0.22302&0.06070 & 0.19529& 0.05909\\
     5& 11& 0.16104& 0.05375& 0.14927& 0.05063\\
    \end{tabular}
    \label{table:lambda_c_values}
    \end{center}
\end{table}

Applying these values to the DoD scheme, we will denote this as the optimized scheme.
Note that this is solely an optimization to balance the operator norm for this specific model case.

\subsection{Exemplary discussion of the optimized $\lambda_c$}\label{sec:05_analysis}
In Figure~\ref{fig:operatornorms_best_lambda_c}, we display the evolution of some operator norms for an optimized value of $\lambda_c$, depending on the chosen basis.

\begin{figure}[htbp]
  \begin{minipage}[t]{0.32\textwidth}
    \centering
    \includegraphics[width=\textwidth]{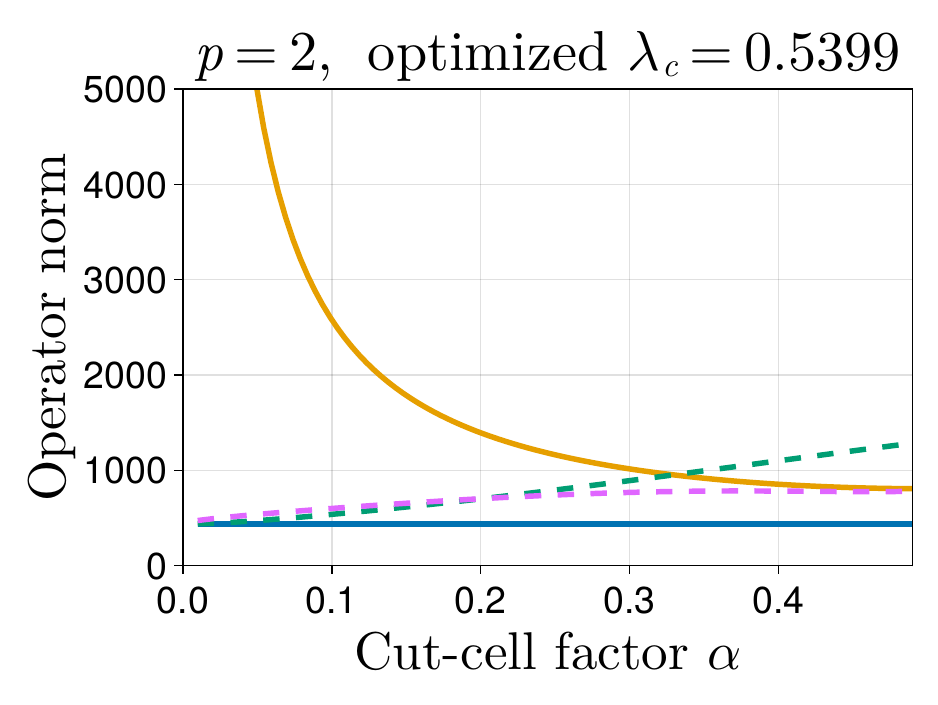}\\
  \end{minipage}
  \begin{minipage}[t]{0.32\textwidth}
    \centering
    \includegraphics[width=\textwidth]{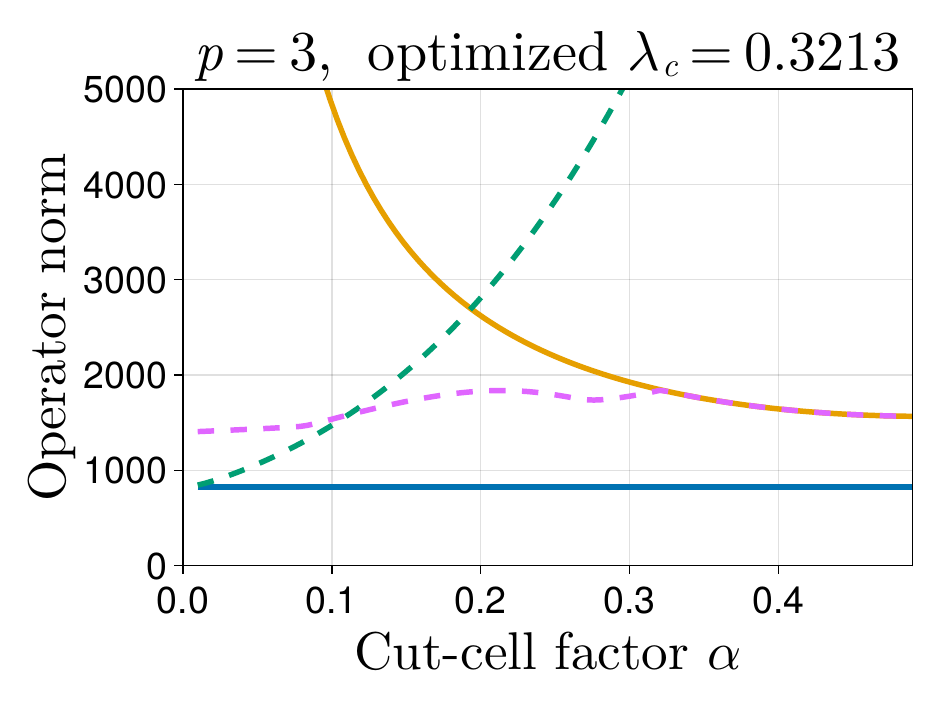}\\
  \end{minipage}
  \begin{minipage}[t]{0.32\textwidth}
    \centering
    \includegraphics[width=\textwidth]{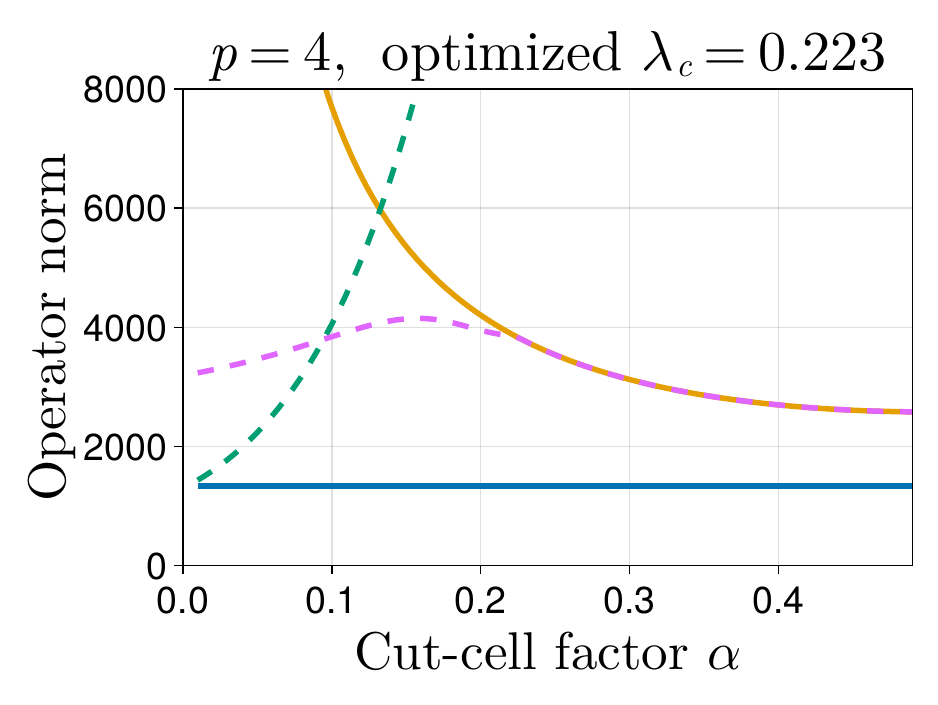}\\
  \end{minipage}
  \centering
  \includegraphics[width=0.6\textwidth, trim ={0 5cm 0 5cm} , clip]{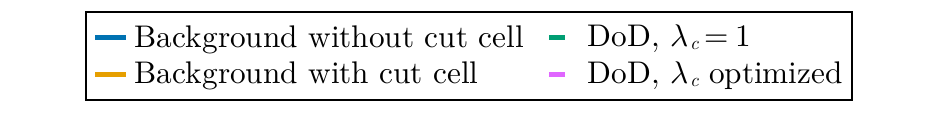}
  \caption{Operator norms for GLB nodes and different orders with $\Delta x = 1/50$.}
  \label{fig:operatornorms_best_lambda_c}
\end{figure}

The concave parts clearly correspond to the damping of the extrapolation terms, while we observe otherwise an increasing trend in $\alpha$.
This also suggests that $4\alpha/\lambda_c^2\|M^{-1}\hb_L\|^2$ is responsible for these parts, as the other terms in \eqref{eq:operatornrom_estimate} that depend on $1/\lambda_c^2$ are constant in $\alpha$ (and retain their dependency on it according to previous estimates).
Note that this optimized choice solely aims to reduce the operator norm and neglects the intention behind the exact physical domain of dependence.
Nevertheless, we still contend with the increased stencil for the outflow cell, albeit with a weighting that does not depend on the physical domain of dependence.
Additionally, in the optimized cases, the DoD operator norm docks onto the unstabilized background norm, which corresponds to turning off the stabilization, i.e., $\eta_c = 0$, as intended.
This indicates that we have effectively addressed the extrapolation aspect, which was never meant to be governed by the minimum function in the definition of $\eta_c$.

\begin{figure}[htbp]
  \centering
    \begin{minipage}[t]{0.33\textwidth}
      \centering
      \includegraphics[width=\textwidth]{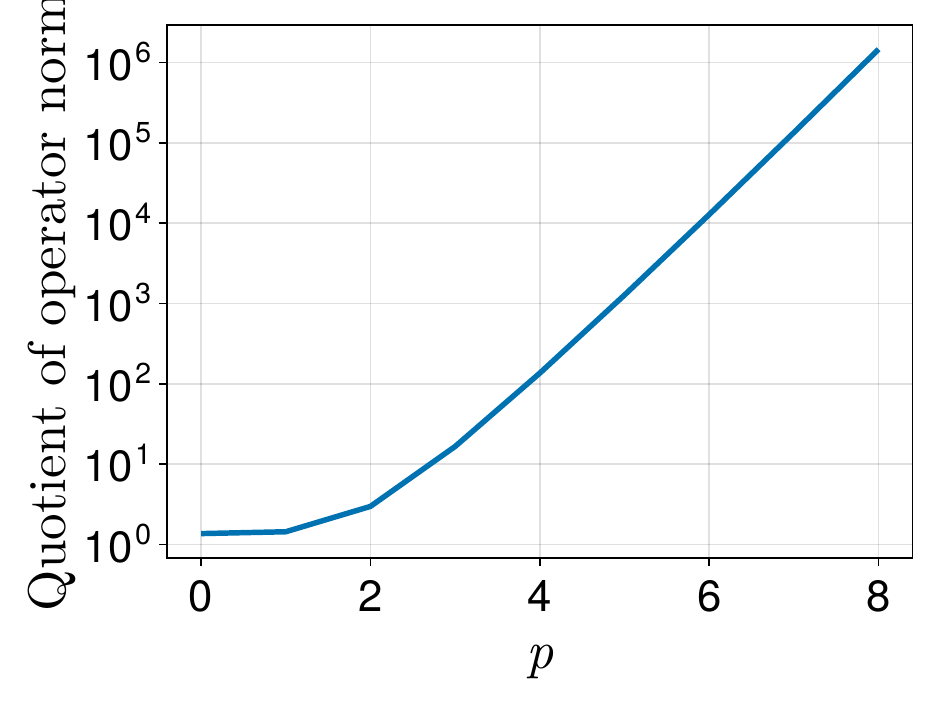}\\
      $\lambda_c = 1 \left( \Rightarrow \eta_c = 1-\alpha \right)$.
    \end{minipage}
    \begin{minipage}[t]{0.33\textwidth}
      \centering
      \includegraphics[width=\textwidth]{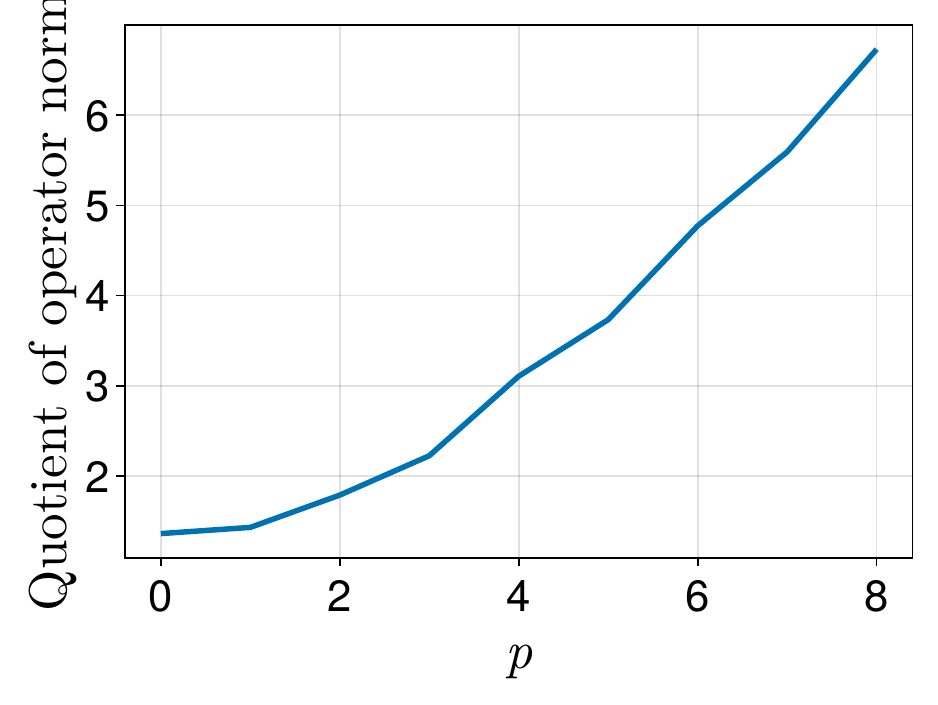}\\
      $\lambda_c$ optimized to achieve a minimized operator norm.
    \end{minipage}
    \caption{Quotients of $\|L_{\text{DoD}}\|/\|L_{\text{background}}\|$ for different choices of $\lambda_c$, using GLL nodes. Not the different scalings of the $y$-axes (logarithmic on the left, linear on the right).}
    \label{fig:operatornorms_quotient}
\end{figure}

In Figure \ref{fig:operatornorms_quotient}, we show the quotient of $\min_{\lambda_c}\|Lu\|$ and the background operator norm for GLL nodes.
We observe that the discrepancy between these norms increases with higher orders, leading to more restrictive time steps as $p$ increases.
However, we can still limit it to some extent compared to previous choices of $\eta_c$.

To provide a deeper practical understanding of the dominating terms that determine
the choice of $\eta_c$, we display in Figure~\ref{fig:partial_operatornorms} the operator norms of the
involved block matrices for the discussed choices of $\lambda_c$ in the specific
cases of GL nodes for $p=4, 5$. Note that we additionally show the operator norm
of the full right-hand side matrix $L$ and $(\star)=S_{c-1}M^{-1}\eta_c\Ip^TD^TM\Ip$. We observe that for $\lambda_c=1$, the terms that contain the
extrapolation operator have a significantly higher norm for larger cut cells
than the remaining ones. We want to emphasize that the norm of $(\star)$, which
applies the extrapolation to the test function and the solution, dominates
$L_{(c-1)}$ and all the other terms. This aligns with our theoretical analysis.
By using the optimized values for $\lambda_c$, we observe that these terms are not
responsible for the major contribution to the norm of $L$, but still represent the
maximum value. Further, we can assume that the terms $L_c$ and $L_{(c-1)}$
contribute the most to $\|L\|_M$, as their operator norms are of the largest
magnitude. This also aligns with our theory, as this optimization corresponds
to reduced stabilization of the small cut cell.

\begin{figure}[htbp]
  \centering
  \begin{minipage}[t]{0.37\textwidth}
    \centering
    \includegraphics[width=\textwidth]{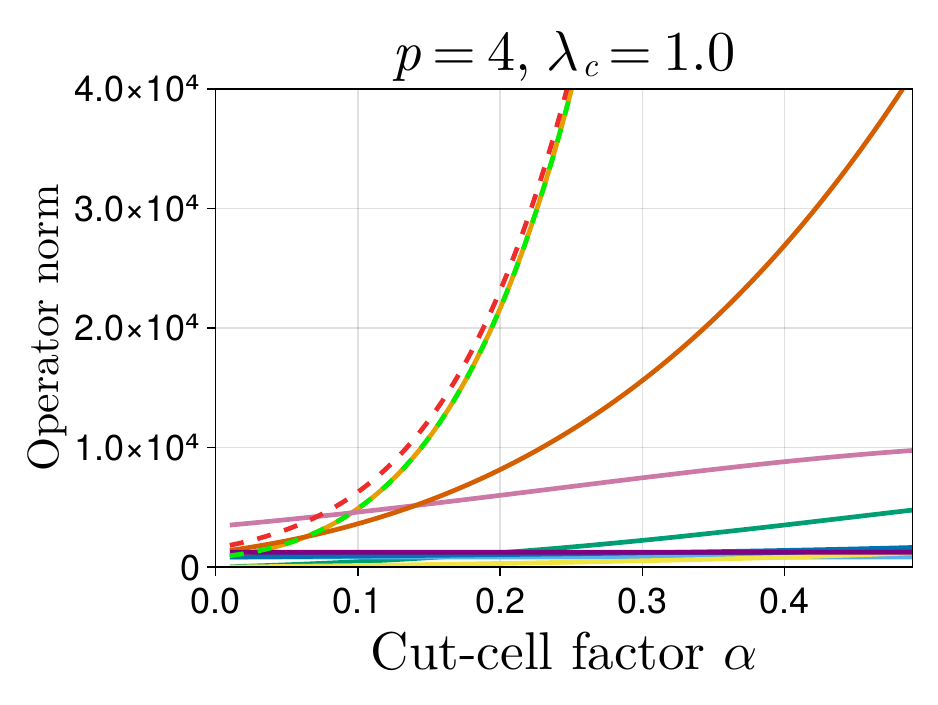}\\
  \end{minipage}
  \begin{minipage}[t]{0.37\textwidth}
    \centering
    \includegraphics[width=\textwidth]{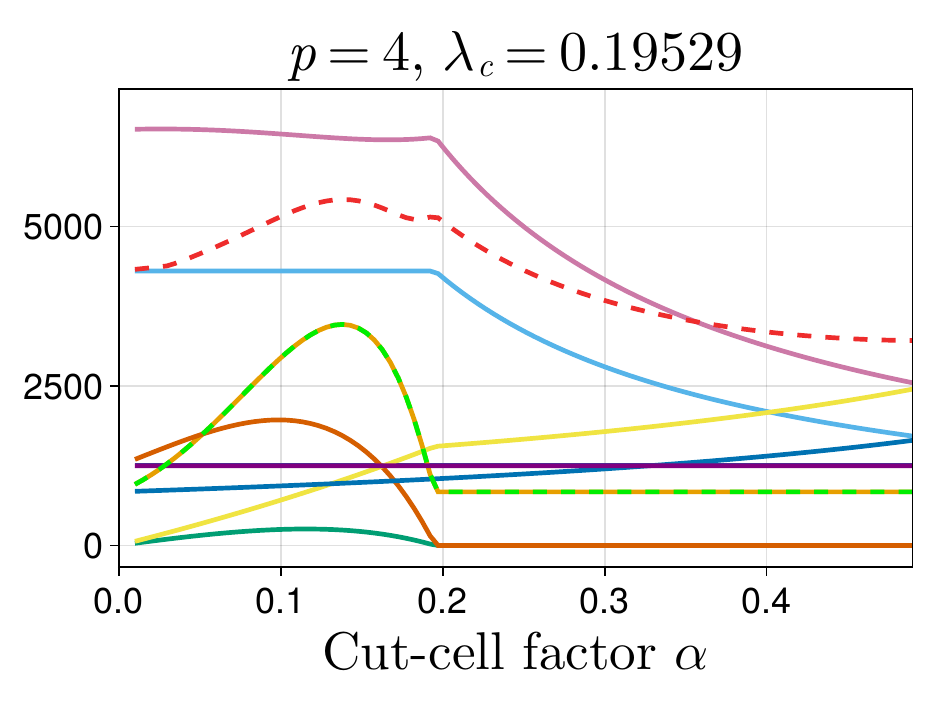}\\
  \end{minipage}\\
  \begin{minipage}[t]{0.37\textwidth}
    \centering
    \includegraphics[width=\textwidth]{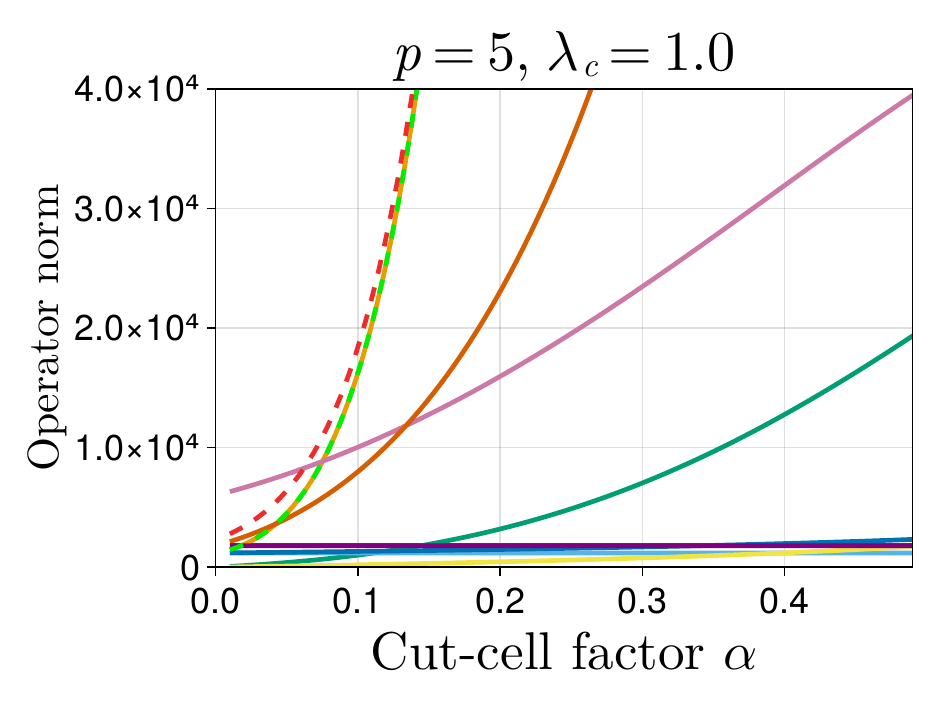}\\
  \end{minipage}
  \begin{minipage}[t]{0.37\textwidth}
    \centering
    \includegraphics[width=\textwidth]{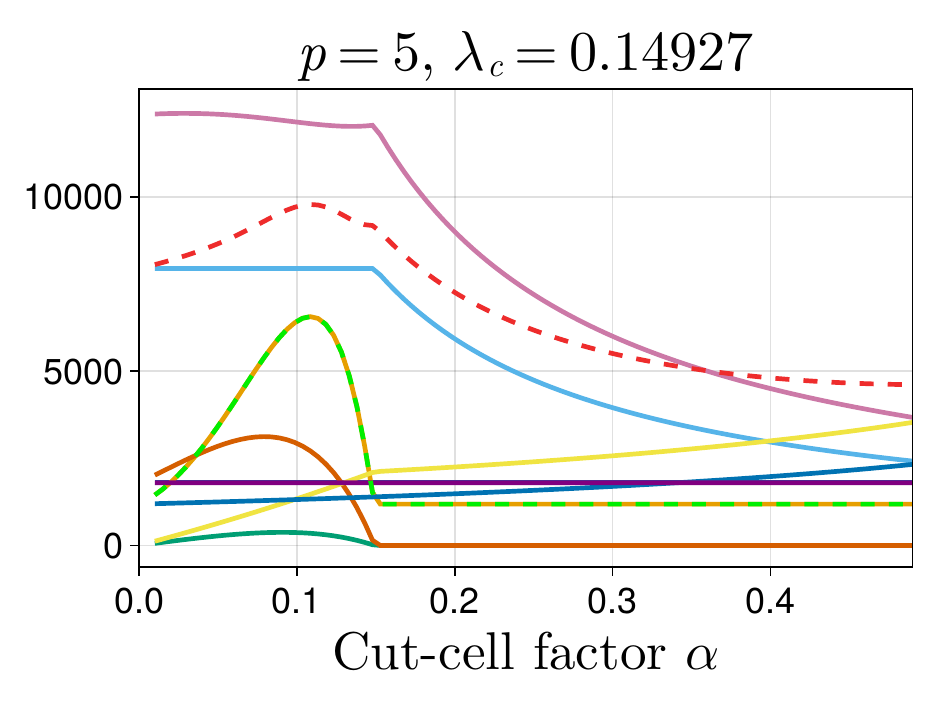}\\
  \end{minipage}
  \centering
  \includegraphics[width=0.7\textwidth, trim ={0 5cm 0 5cm} , clip]{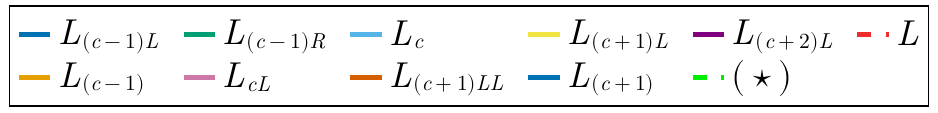}
  \caption{Operator norms of the involved block matrices in \eqref{eq:dod_semidiscretization_matrix}, including $L$ and the
   partial term $(\star)=S_{c-1}M^{-1}\eta_c\Ip^TD^TM\Ip$ of $L_{(c-1)}$. Using Gauss-Legendre nodes and $\Delta x = 1/50$.}
  \label{fig:partial_operatornorms}
\end{figure}

%% file: 06-1_numerics1D.tex
In this section, we present numerical tests of the linear transport equation to validate our theoretical findings.
We begin with the one-dimensional linear advection equation as in the
previous analysis and then proceed with some applications in 2D.

The 1D numerical methods are implemented in Julia
\cite{bezanson2017julia} based on the infrastructure of \newline
\texttt{SummationByPartsOperators.jl}~\cite{ranocha2021sbp}.

The 2D tests are implemented using the C++ PDE framework
DUNE~\cite{dune-recent,dunepaperI:08,dunepaperII:08,dune2.10}, including
the dune-functions module for functions space
bases~\cite{dune-functions-1,dune-functions-2}. We further rely on the
TPMC \cite{tpmc} library to calculate the cut cell geometries. The presented 2D data was generated with the help of GNU Parallel \cite{tange}.

We used Makie.jl \cite{danisch2021makie} to generate the plots.  All
code and data required to reproduce the numerical results presented in
this article are available in our reproducibility repository
\cite{petri2025domainRepro}.

\subsection{One-dimensional linear advection}
\label{sec:06_1D_numerics}

To investigate this optimization in practice, we apply the fully discrete scheme to the linear advection equation and determine the sharp CFL condition by adjusting the time step under which the method remains stable in long-term simulation.
We will compare this to the classic case with $\lambda_c=1$ and the background method, i.e., the DGSEM without any cut cells in the mesh.
We use the same settings for the semidiscretization as above and again vary $\alpha$ in $[0.001, 0.49]$.
We choose $a=1$ and $\sin(2 \pi x)$ as the initial condition over the domain $[0,1]$.
For time-stepping, which affects the explicit results, we apply appropriate
explicit RK methods that correspond to the order of the semidiscretization,
specifically SSPRK(2,2), SSPRK(3,3) \cite{Gottlieb1996},
and SSPRK(10,4) \cite{Ketcheson2008ssp}.
For the first order, we tested the explicit Euler method and obtained $\lambda_c=1$; there are clearly no differences compared to DoD with $\eta_c=\eta_c(\Delta t)$.
In each case, we obtain the sharp CFL restriction of 1 (independent of $\alpha$), which is also recognized from the background method.
Consequently, we do not present a plot for this case. The remaining results, up to order 4, are illustrated in Figure~\ref{fig:optimal_cfl_lts_1D}.

\begin{figure}
    \begin{minipage}[t]{0.32\textwidth}
      \centering
      \includegraphics[width=\textwidth]{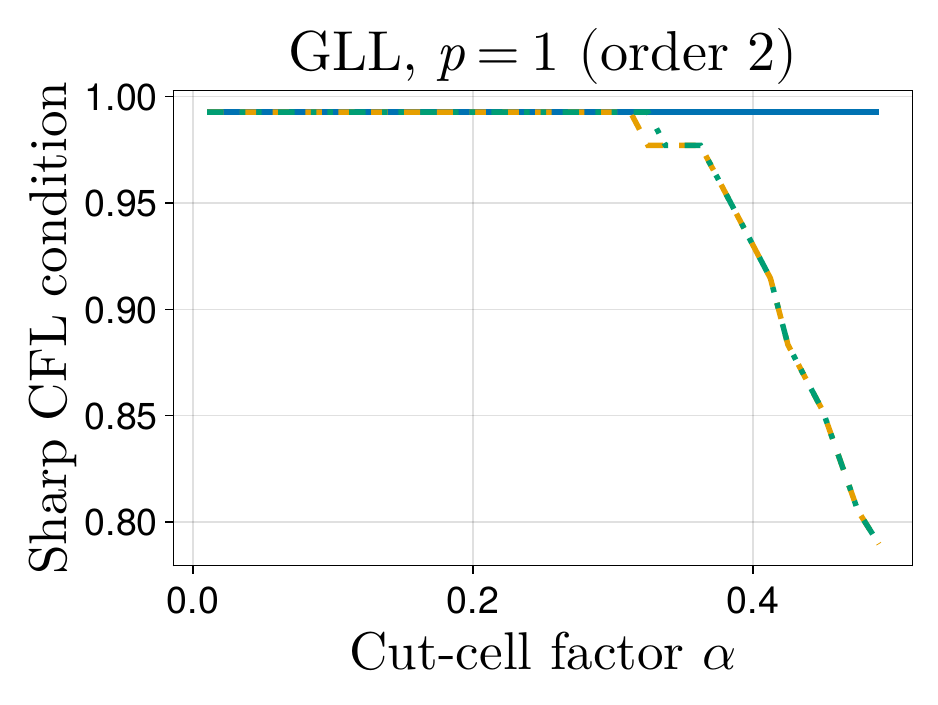}\\
    \end{minipage}
    \begin{minipage}[t]{0.32\textwidth}
      \centering
      \includegraphics[width=\textwidth]{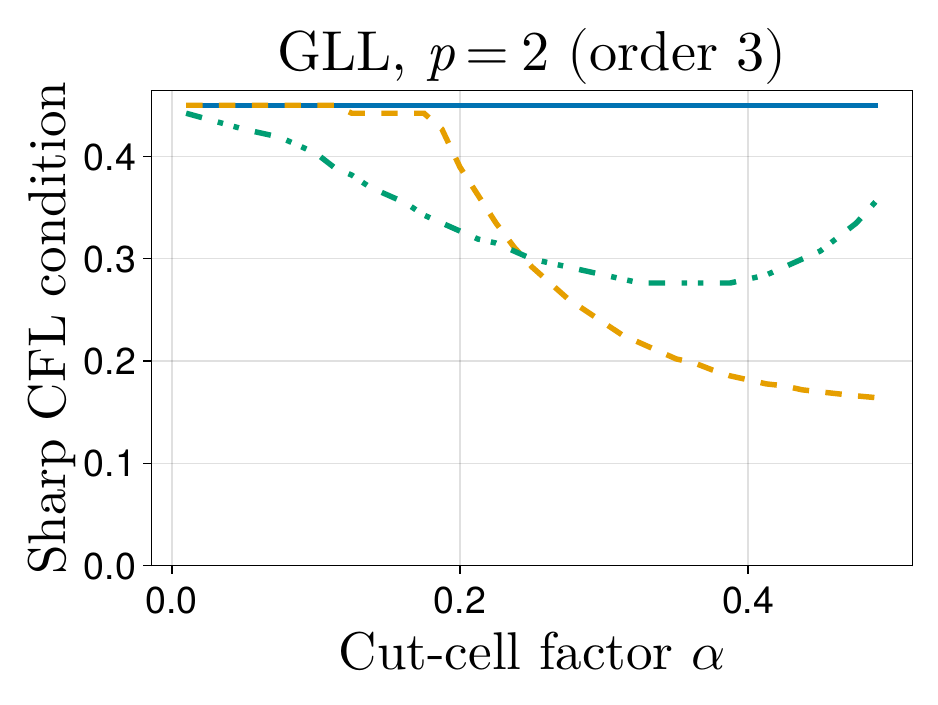}\\
    \end{minipage}
    \begin{minipage}[t]{0.32\textwidth}
      \centering
      \includegraphics[width=\textwidth]{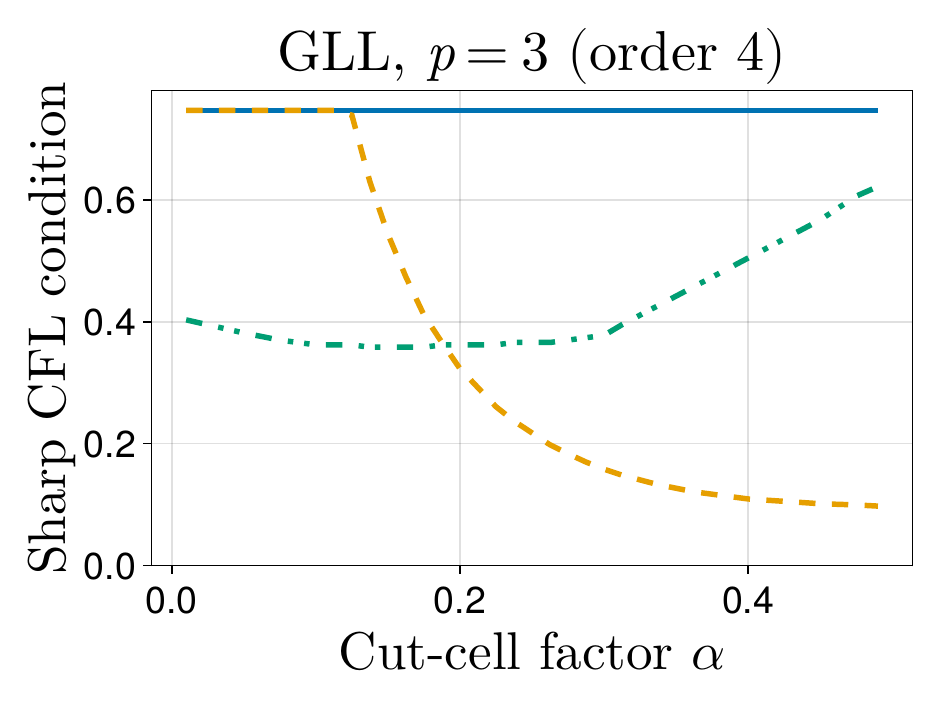}\\
    \end{minipage}\\
    \begin{minipage}[t]{0.32\textwidth}
      \centering
      \includegraphics[width=\textwidth]{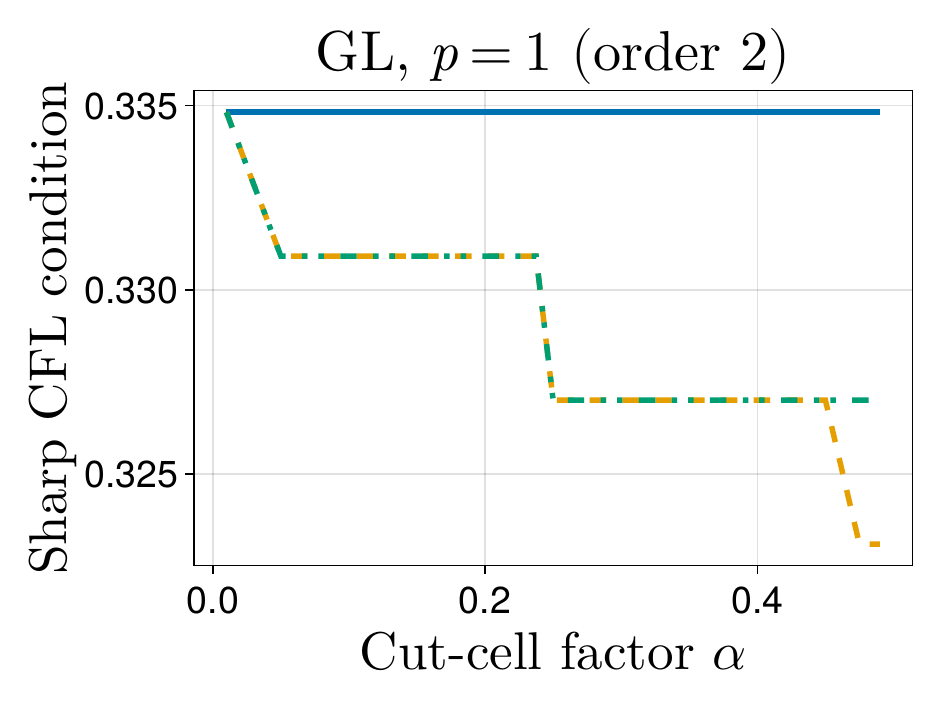}\\
    \end{minipage}
    \begin{minipage}[t]{0.32\textwidth}
      \centering
      \includegraphics[width=\textwidth]{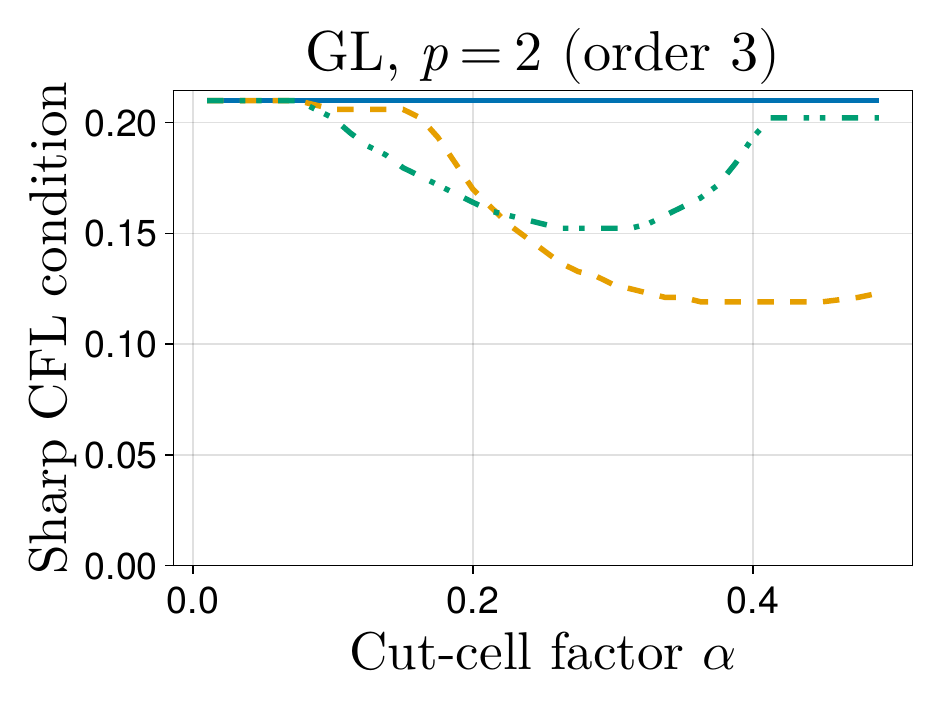}\\
    \end{minipage}
    \begin{minipage}[t]{0.32\textwidth}
      \centering
      \includegraphics[width=\textwidth]{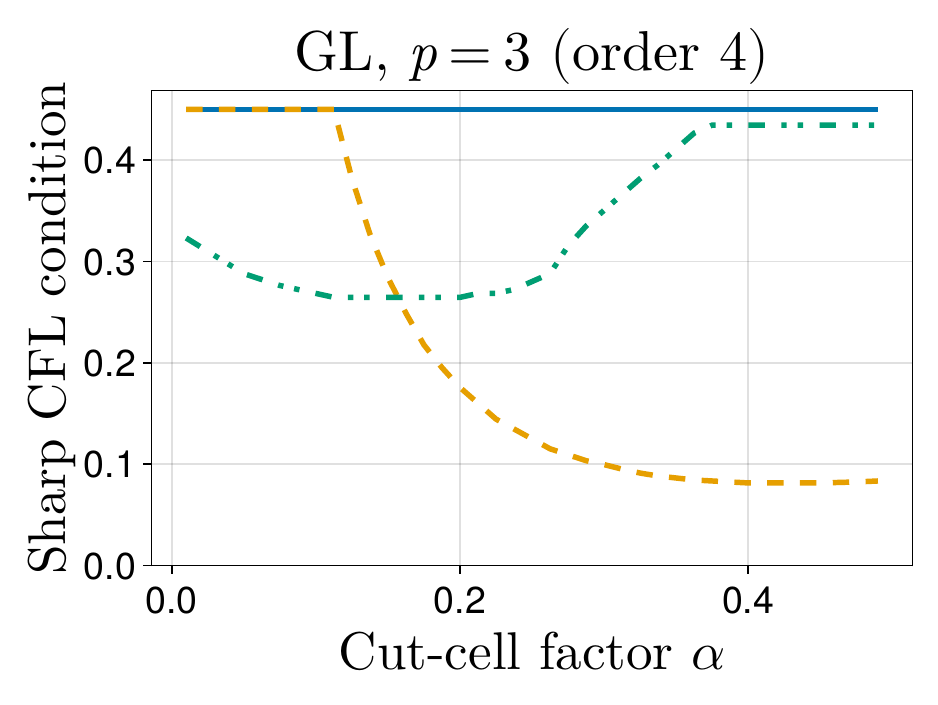}\\
    \end{minipage}
    \centering
    \includegraphics[width=0.7\textwidth, trim ={0 5.5cm 0 5.5cm} , clip]{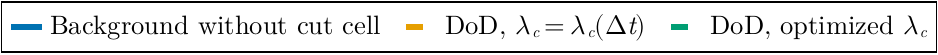}
    \caption{Sharp CFL restriction for long time stability, depending on the cut cell size.}
    \label{fig:optimal_cfl_lts_1D}
\end{figure}

We observe a similar behavior for both types of nodes: For the second order methods, we still see barely a difference between the two choices. Considering the higher orders,
the time step dependent DoD required a slightly loose CFL condition for small cut cells, as it strongly shifts for the larger cut cells. We also observe, that this gets more crucial, the higher the order
of the scheme. This is exactly what we would expect, as small values for $\lambda_c$ reduce the amount of stabilization, that is necessary for small cut cells and tends to be problematic for larger ones.
As in ``practice'', there are cut cells of arbitrary sizes. Therefore, the question of which parameters to choose reduces also to the highest possible CFL, that does not depend on a specific cut cell size, i.e.
we aim to choose the lowest possible CFL value for a specific choice of $\lambda_c$. In that case, our optimal choice seems to be more suitable and will be even more for higher order methods for this model problem.
Please note that the overall increased CFL values for the fourth-order case compared to the third-order case are due to the usage of the SSPRK(10,4), which uses more stages than necessary to achieve fourth-order accuracy to improve stability properties.

\begin{figure}
  \begin{minipage}[t]{0.32\textwidth}
    \centering
    \includegraphics[width=\textwidth]{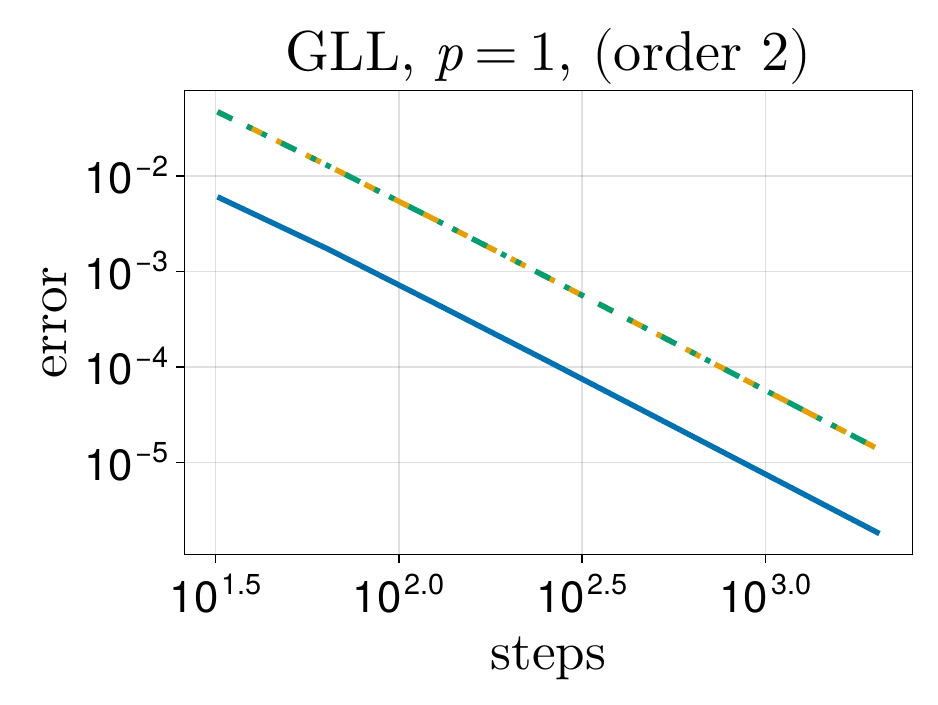}\\
  \end{minipage}
  \begin{minipage}[t]{0.32\textwidth}
    \centering
    \includegraphics[width=\textwidth]{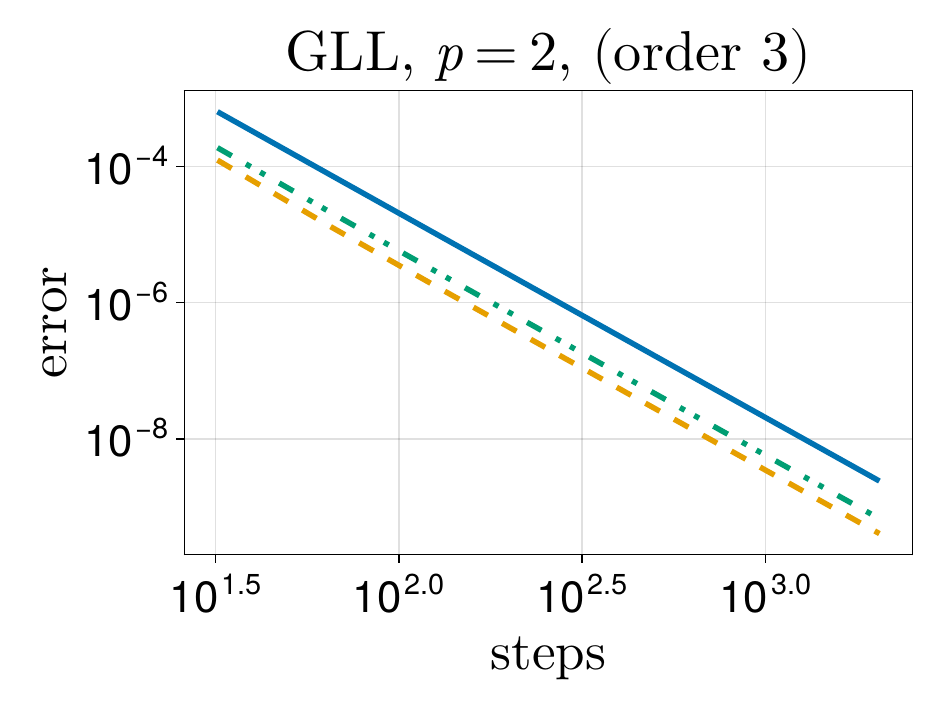}\\
  \end{minipage}
  \begin{minipage}[t]{0.32\textwidth}
    \centering
    \includegraphics[width=\textwidth]{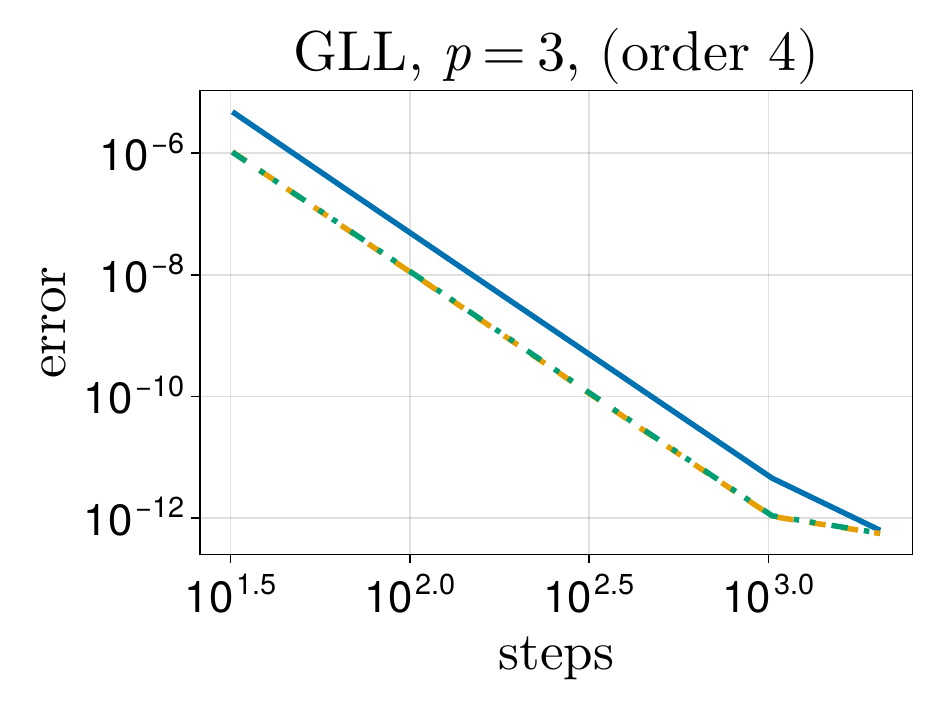}\\
  \end{minipage}\\
  \begin{minipage}[t]{0.32\textwidth}
    \centering
    \includegraphics[width=\textwidth]{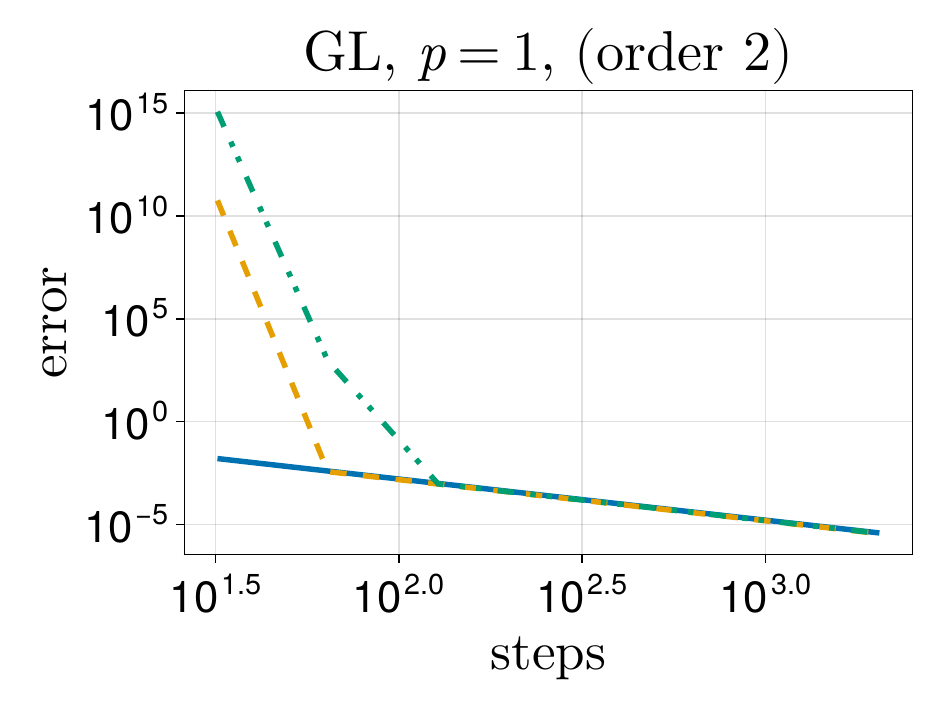}\\
  \end{minipage}
  \begin{minipage}[t]{0.32\textwidth}
    \centering
    \includegraphics[width=\textwidth]{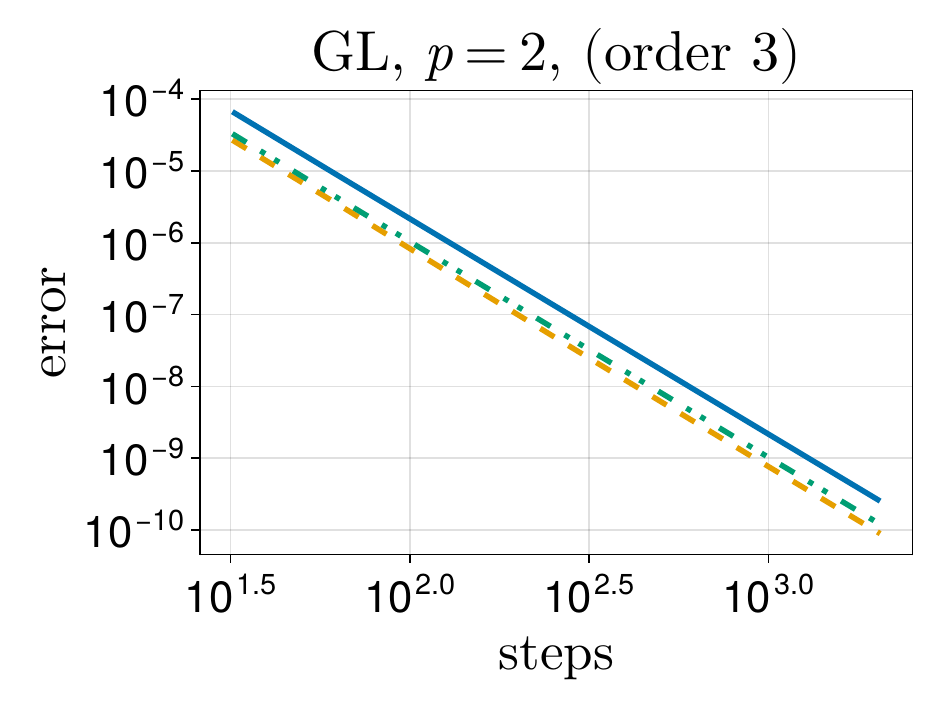}\\
  \end{minipage}
  \begin{minipage}[t]{0.32\textwidth}
    \centering
    \includegraphics[width=\textwidth]{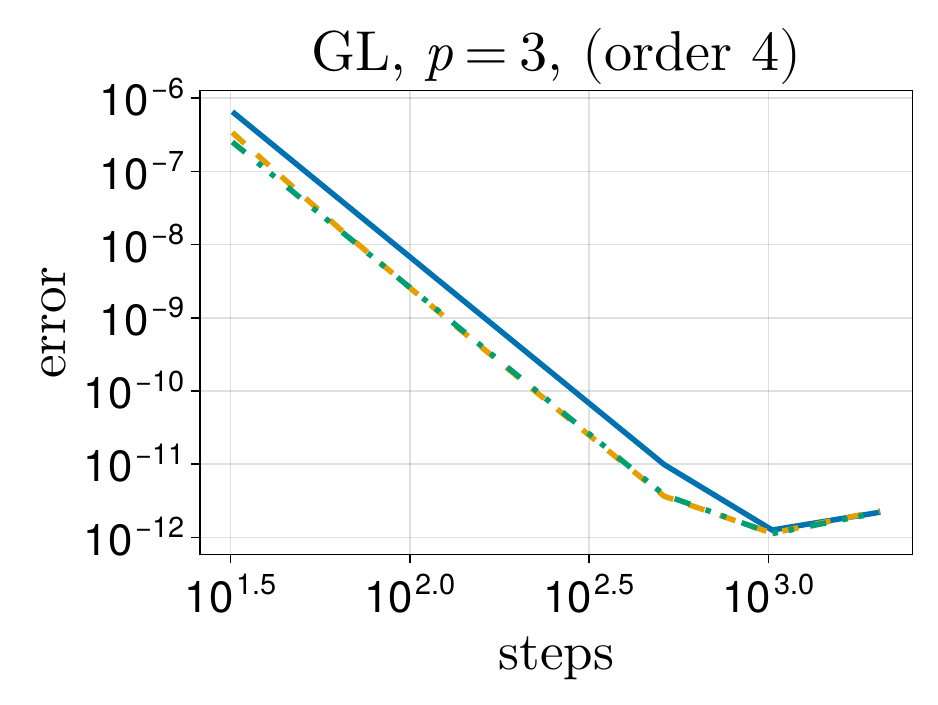}\\
  \end{minipage}
  \centering
  \includegraphics[width=0.7\textwidth, trim ={0 5.5cm 0 5.5cm} , clip]{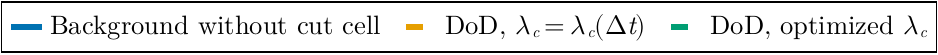}
  \caption{Convergence errors at the time $T=1$ with respect to the norm $\|\cdot\|_M$.}
  \label{fig:convergence_1D}
\end{figure}

Next, we want to investigate these methods for numerical convergence with the sharp CFL conditions.
Specifically, we select the largest sharp CFL, such that all considered cases (the background method and the two DoD variants)
 for fixed nodes and polynomial degree are stable and refine the mesh by applying multiple cut cells, ranging from $\alpha = 10^{-3}$ to $\alpha = 0.49$.
We also apply a safety factor of $0.95$ to the respective sharp CFL value to ensure that minor calculation errors do not impact stability.
As shown in Figure~\ref{fig:convergence_1D}, convergence is achieved as expected.
In Figure~\ref{fig:work_precision_1D}, we illustrate the error in relation to the number of time steps leading up to the final time of $T=1$, which correlates with the efficiency of the different schemes.
For this case, we select the corresponding maximal sharp CFL number for which the method is stable and apply a safety factor of $0.99$.
Here, we observe that the optimization results in at least the same and often smaller errors than the ``classic'' method,
indicating that this operator norm-based choice not only contributes to a more favorable CFL condition but also impacts the actual accuracy.
Once again, please note that these results pertain to the simple model problem for now.

\begin{figure}
  \begin{minipage}[t]{0.32\textwidth}
    \centering
    \includegraphics[width=\textwidth]{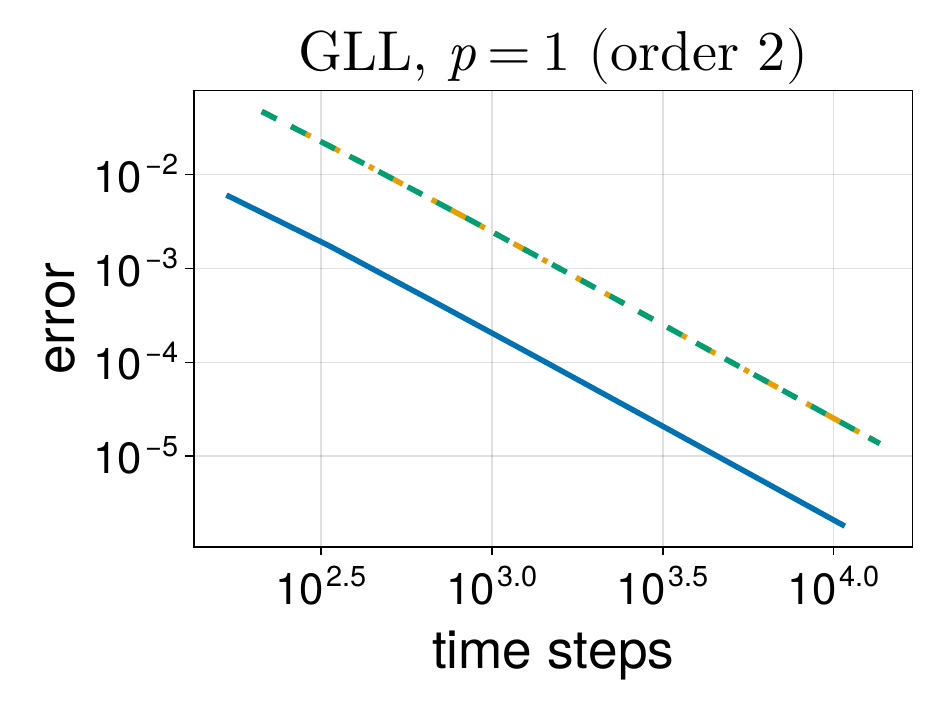}\\
  \end{minipage}
  \begin{minipage}[t]{0.32\textwidth}
    \centering
    \includegraphics[width=\textwidth]{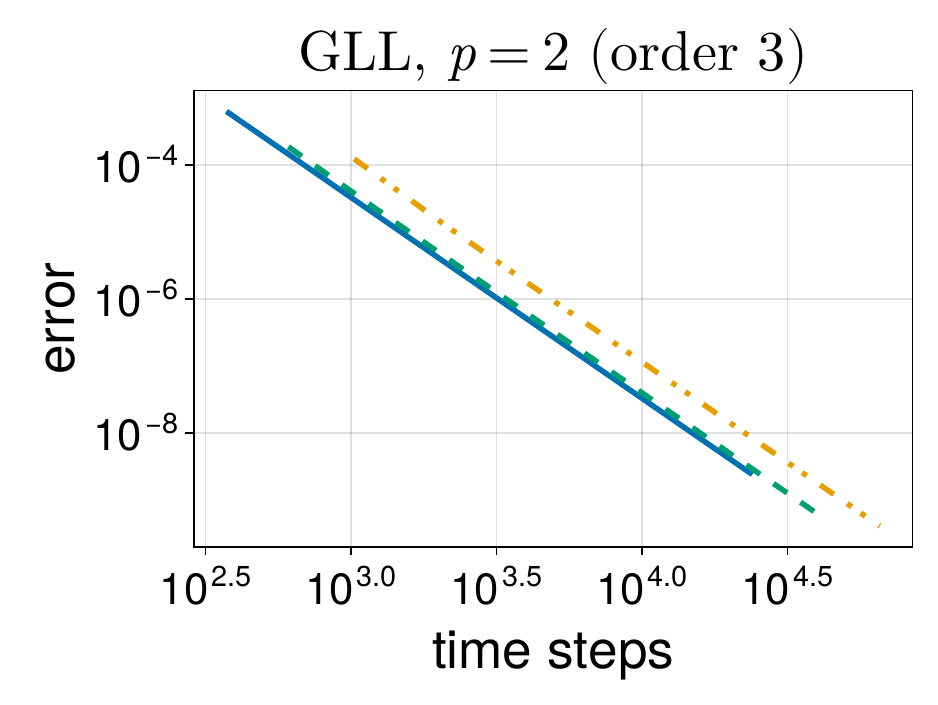}\\
  \end{minipage}
  \begin{minipage}[t]{0.32\textwidth}
    \centering
    \includegraphics[width=\textwidth]{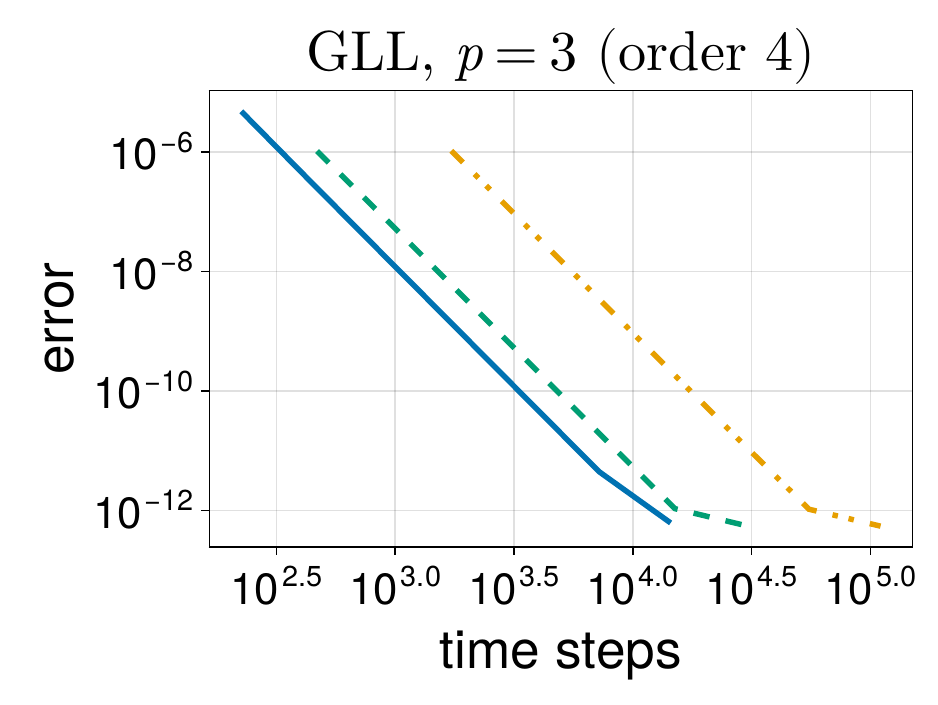}\\
  \end{minipage}\\
  \begin{minipage}[t]{0.32\textwidth}
    \centering
    \includegraphics[width=\textwidth]{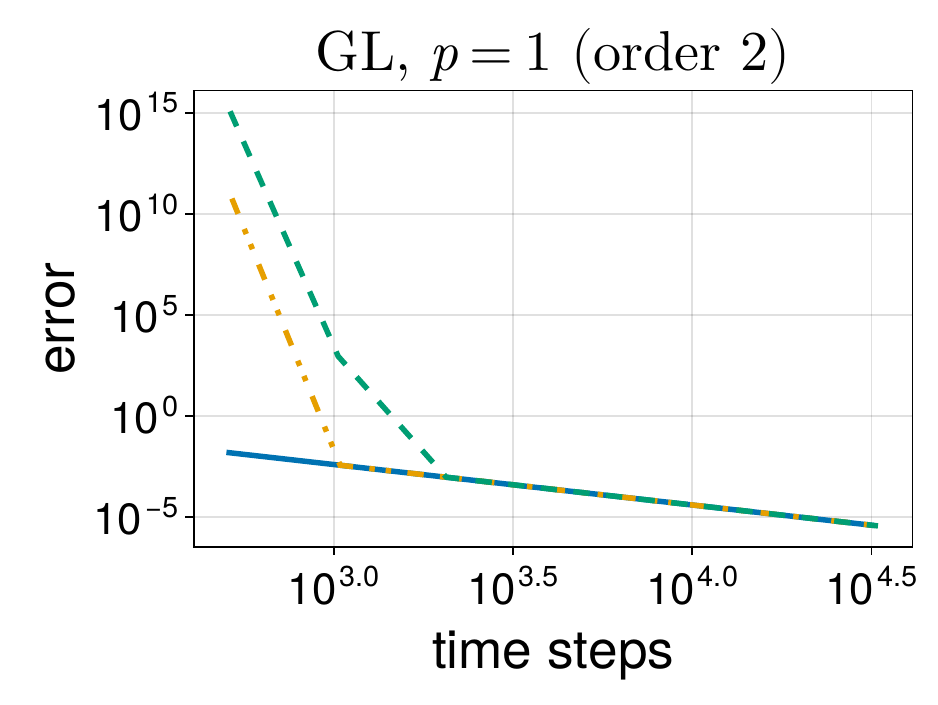}\\
  \end{minipage}
  \begin{minipage}[t]{0.32\textwidth}
    \centering
    \includegraphics[width=\textwidth]{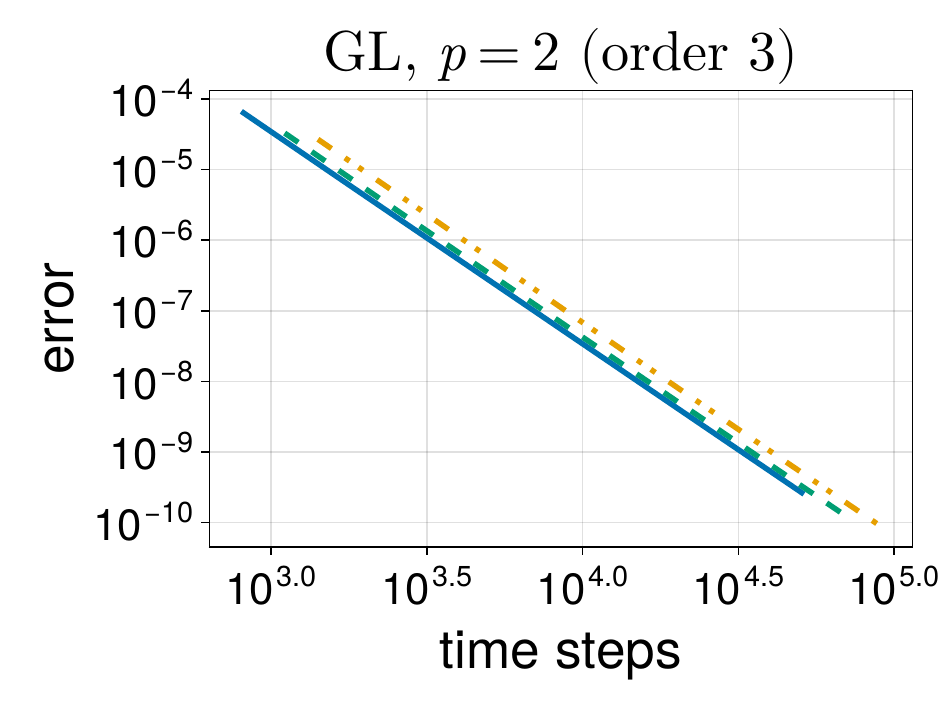}\\
  \end{minipage}
  \begin{minipage}[t]{0.32\textwidth}
    \centering
    \includegraphics[width=\textwidth]{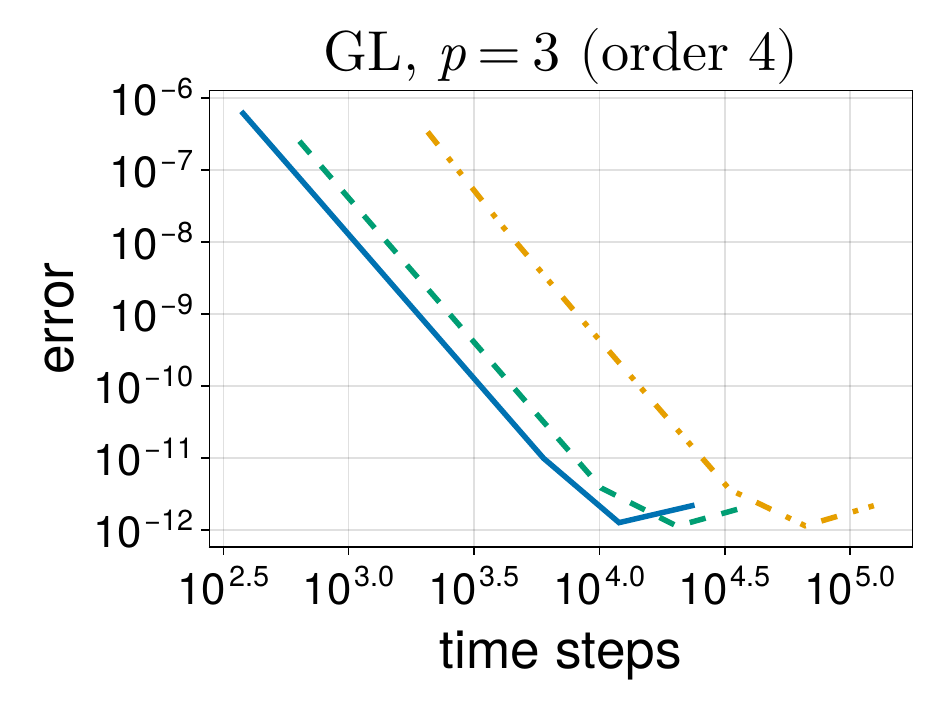}\\
  \end{minipage}
  \centering
  \includegraphics[width=0.7\textwidth, trim ={0 5.5cm 0 5.5cm} , clip]{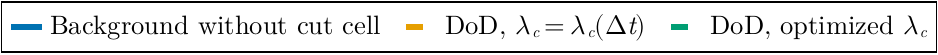}
  \caption{Work precision diagram with errors at the time $T=1$ with respect to the norm $\|\cdot\|_M$.}
  \label{fig:work_precision_1D}
\end{figure}

\subsection{Two-dimensional linear advection in a channel}

\begin{figure}[htbp] \centering
	\begin{tikzpicture}[scale=5.0]
		\draw[step = 0.05] (0.0, 0.0) grid (1.0, 0.7);

		\draw (0.101, 0.0) -- (1.0, 0.4495);
		\draw (0.0,  0.101) -- (1.0, 0.601);

		\draw[fill=gray] (0.101, 0.0) -- (1.0, 0.4495) -- (1.0, 0.0) -- cycle;
		\draw[fill=gray] (0.0,  0.101) -- (1.0, 0.601) -- (1.0, 0.7)  -- (0.0, 0.7) -- cycle;

		\draw[->] (-0.5, 0.2) -- (-0.25, 0.325) node [midway, below] {$\beta$};
	\end{tikzpicture}
	\caption{Channel geometry with cut-cell mesh and advection velocity $\beta$.}
	\label{fig: channel geometry}
\end{figure}
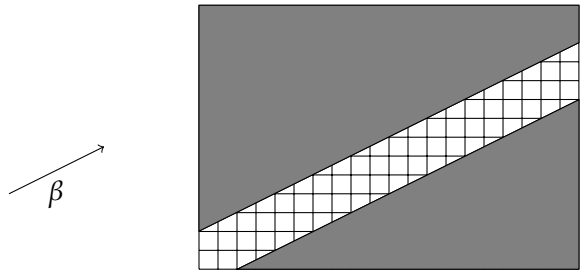

We consider the linear advection equation
\begin{align*}
	\partial_t u(t, x) + \beta \cdot \nabla_x u(t, x) & = 0 & \text{in } \Omega \subset \mathbb{R}^2,\\
	u(0, x) & = u_0(x) & \text{in } \Omega,\\
	u(t, x) & = 0 & \text{on } \partial \Omega,
\end{align*}
with homogeneous Dirichlet data in a channel domain with parallel walls, see Figure~\ref{fig: channel geometry} for a graphical representation. The lower, respectively the upper wall, will start at a position $(x_0, 0)^t$, respectively $(0, x_0)^t$ with an offset $x_0 \in \mathbb{R}$. The geometry is completely defined by the choice of $x_0$ and the angle $\gamma$ between the walls and the $x$-axis. The flow direction $\beta$ is a constant unit vector parallel to the channel walls.

First, we consider the choice $\gamma = 45^{\circ}$. This creates a geometric setup where all small cut cells in need of stabilization share the same size and shape, similar to the one-dimensional case. We can define the size of these small cells, i.e., the cut-cell factor $\alpha$, by adjusting the offset $x_0$, which allows us to run numerical tests with a single size factor.

We partially replicate the 1D analysis of the operator norm given in \eqref{eq:minmax_opnorm} in 2D by estimating an optimal value for the parameter $\lambda_c$ from \eqref{eq:penalty_parameter_defi} for different polynomial degrees.  The results are presented in Figure~\ref{fig:operator norms 2d}, where we plotted different choices of $\lambda_c$ over a variety of cut-cell factors. In particular, we plotted a choice of $\lambda_c$ which is close to the optimal factors presented in the previous section for the one-dimensional case. As seen in the plots these values are not optimal anymore for the two-dimensional channel geometry. However, we were able to obtain new optimal values for $\lambda_c$, which we included in the plots. We note that the function graphs of these new values possess a shape with two peaks similar to the one-dimensional case, highlighting the optimization of different contributions of the scheme.

\begin{figure}[ht]
	\centering
\begin{tikzpicture}[scale=3.5]
\draw[gray, thin] (0.0, 1.0) -- (1.0, 1.0) -- (1.0, 0.0) -- (0.0, 0.0) -- cycle;
\draw[dashed, thick, fill = blue, fill opacity=0.2] (0.0, 1.0) -- (0.9, 1.0 ) -- (0.0, 0.1) -- cycle;
\draw[thick, pattern = dots, pattern color = red] (0.0, 1.0) -- (0.9, 1.0 ) -- (0.0, 0.55) -- cycle;

\draw (0.6, 1.0) arc (180:210:0.25);
\draw (0.45, 1.0) arc (180:225.0:0.45);

\node at (0.69, 0.94) {$\gamma$};
\node at (0.61, 0.79) {$45^\circ$};

\node at (0.2, 0.8) {$E_{\gamma}$};
\node at (0.3, 0.55) {$E_{45^\circ}$};

\node at (-0.15, 0.5) {$e_{\text{in}}$};
\node at (0.5, 1.1) {$e_{\text{out}}$};

\draw[->] (-0.5, 0.5) -- (-0.25, 0.625) node [midway, below] {$\beta^{\gamma}$};

\draw[->] (-0.5, 0.0) -- (-0.25, 0.25) node [midway, below] {$\beta^{45^\circ}$};
\end{tikzpicture}
\caption{Comparison of two triangular cut cells, $E_{45^\circ}$ in blue and $E_\gamma$ with red dots, with different cut angles and flow directions (that are always parallel to the cut). The distance in the normal direction of the inflow face $e_{\text{in}}$ stays the same for both cut cells, no matter the cut angle, so we would expect that in both cases the same amount of stabilization is necessary,  at least if the cut interface is a no-flow face (which we assume to be the case here). Note that we have $ |E_\gamma| = \tan(\gamma) |E_{45^\circ}|$ and thus $ \alpha_\gamma = \tan(\gamma) \alpha_{45^\circ}$ where $\alpha_{\gamma}$ and $\alpha_{45^\circ}$ are the cut-cell factors of $E_\gamma$ and $E_{45^\circ}$, respectively.}
\label{fig: angle dependence of capacity}
\end{figure}

As a first step towards a truly multidimensional setting, we consider additional angles of $\gamma$ for the channel geometry. Once we move away from the case $\gamma = 45^{\circ}$ the cut-cell structure of the geometry becomes more rich and direct control over cell volume fractions is lost. Therefore, the previous optimization procedure is not directly applicable anymore. Based on geometrical considerations (see Figure~\ref{fig: angle dependence of capacity} for an explanation) we instead suggest the general choice
 \begin{equation}
 	\label{eq: gamma_c choice}
 \lambda_c^{\gamma} = \lambda_c^{45^{\circ}}  \tan(\gamma)
 \end{equation}
 for $\gamma \in (0^{\circ}, 45^{\circ}]$ and $\lambda_c^{45^{\circ}}$ the optimal value already obtained and displayed in Figure~\ref{fig:operator norms 2d}. Note that the limit case of $\gamma = 0^{\circ}$ is just a Cartesian grid and requires no stabilization.

We can still compare this choice with an optimal value for $\lambda_c$ obtained over the same offsets $x_0$ that we used for the $\gamma = 45^{\circ}$ channel (even though we want to stress here again, that now these offsets in general do not correspond to cell volume fractions). The results of this comparison are compiled in Table~\ref{tab: operator norms 2d_2}. In general, we can see good agreement between the suggested choice and the optimal value.

Equipped with a good $\lambda_c$ for different polynomial degrees and channel angles we can give numerical estimates of CFL numbers. These are presented in Table~\ref{tab: cfl numbers 2d} and correspond to typical CFL numbers based on the background mesh for DG schemes in combination with our chosen time integration methods.

\begin{figure}
	\begin{minipage}[t]{0.32\textwidth}
		\centering
		\includegraphics[width=\textwidth]{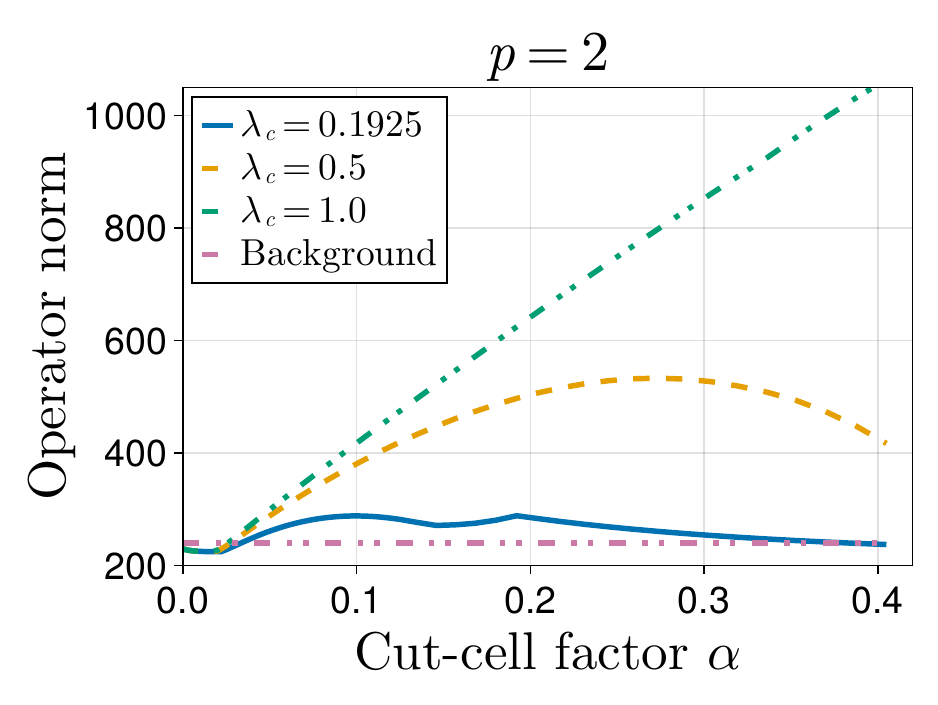}\\
	\end{minipage}
		\begin{minipage}[t]{0.32\textwidth}
		\centering
		\includegraphics[width=\textwidth]{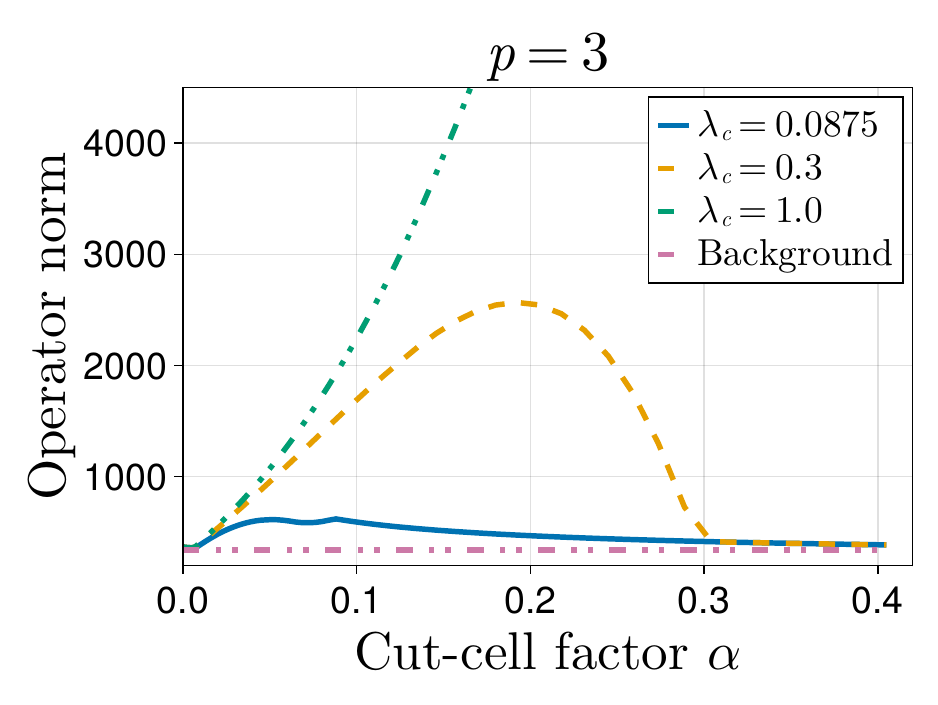}\\
	\end{minipage}
		\begin{minipage}[t]{0.32\textwidth}
		\centering
		\includegraphics[width=\textwidth]{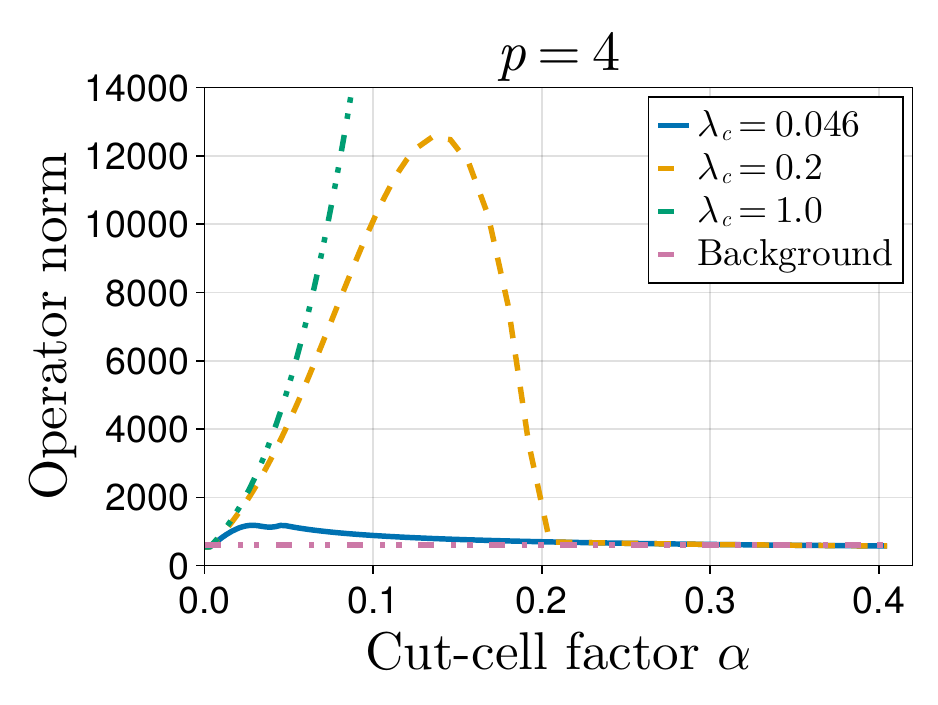}\\

	\end{minipage}
	\caption{Operator norms for polynomial degrees $p=2, 3, 4$ and different choices of $\lambda_c$ on the channel geometry.}
	\label{fig:operator norms 2d}
\end{figure}

\begin{table}
	\caption{Comparison of the choice of $\lambda_c$ for different angles $\gamma$.}
	\label{tab: operator norms 2d_2}
	\centering
	\begin{tabular}{@{~}c@{~}|| c | c | c | c || c | c | c | c || c | c  | c | c }
		& \multicolumn{4}{c||}{$p=2$} & \multicolumn{4}{c||}{$p=3$} & \multicolumn{4}{c}{$p=4$}\\
		\cline{2-13}
		& \!\scriptsize$\lambda_c\tan(\gamma)$\! & $\norm{L}$ & opt. & $\norm{L}$ & \!\scriptsize$\lambda_c\tan(\gamma)$\! & $\norm{L}$ &  opt. & $\norm{L}$ & \!\scriptsize$\lambda_c\tan(\gamma)$\! & $\norm{L}$ &  opt. & $\norm{L}$  \\
		\hline
		$35^{\circ}$& 0.135 & 316 & 0.161 & 299 & 0.061 & 683 & 0.067 & 667 & 0.032 & 1355 & 0.033 & 1338\\
		\hline
		$25^{\circ}$& 0.090 & 342 & 0.107 & 322 & 0.041 & 761 & 0.045& 734&  0.021 & 1496 & 0.023 & 1460 \\
		\hline
		$15^{\circ}$& 0.052 & 360 & 0.062 & 339 & 0.024 & 798 & 0.027 & 787 & 0.012 & 1586 & 0.012 & 1586 \\
		\hline
		$5^{\circ}$& 0.017 & 368 & 0.020 & 347 & 0.008 & 815 & 0.009 & 803 & 0.004 & 1588 & 0.004 & 1588
	\end{tabular}
\end{table}

\begin{table}
	\caption{CFL numbers achieved for different channel geometries and polynomial degrees $p$ with the optimal parameter $\lambda_c$ given by \eqref{eq: gamma_c choice}.}
	\label{tab: cfl numbers 2d}
	\centering
	\begin{tabular}{c || c | c | c }
		& $p=1$ & $p=2$ & $p=3$\\
		\hline
		$45^{\circ}$ & $0.35$ & $0.24$ & $0.47$\\

		\hline
$35^{\circ}$  & $0.38$  &  $0.24$ & $0.52$\\
\hline
$25^{\circ}$ & $0.37$ &  $0.23 $ & $0.47$\\
\hline
$15^{\circ}$ & $0.37$ &  $0.21$ & $0.45$\\
\hline
$5^{\circ}$ & $0.37$  & $0.21 $ & $0.43$\\
	\end{tabular}
\end{table}

%% file: 07_summary.tex
In this work, we have given a proof that the domain-of-dependence (DoD) stabilization
is fully-discrete stable for a specific class of explicit and arbitrarily
high-order Runge-Kutta methods. To perform this analysis, we rewrote the
semidiscretization as a system of ODEs within the DGSEM framework.
In this setting, we also managed to formulate this result for the case
of under-integration with Gauss-Lobatto-Legendre nodes and quadrature. In particular,
the stability question reduced to an operator norm estimate, and we
demonstrated that this operator norm, and therefore the resulting
CFL-like condition, scale without a problematic dependence on the arbitrarily
small cut cells, as intended.

We discovered further CFL restrictions that appear for larger cut cells
treated with the stabilization, especially in higher-order cases.
Our analysis shows that these restrictions originate from the extension operator,
which projects polynomials outside their assigned cell. We also quantified this
issue with an operator norm estimate. Because this unintended behavior bypassed
the implemented stabilization cutoff, we suggested a practical fix by selecting
the penalty parameter to minimize the operator norm across all potential cut-cell sizes.
This effectively results in an optimal CFL-like condition and restores the useful cutoff.
Numerical examples are provided to illustrate and support our theoretical findings.

As an outlook, we also provide a numerical extension to the case of advection
in two spatial dimensions, where we can observe the same phenomenon in a
similar form. Because of the arbitrary shape of the cut cells, this issue becomes
more intricate. Nevertheless, these first results indicate, that by a 
suitable choice of parameters, the time-step restriction remains in the same order
of magnitude as in a respective background scheme.
A solution that combines the approach of the 1D case with
geometric ideas leads to a promising initial result.

%% file: funding.tex
LP and HR were supported by the Deutsche Forschungsgemeinschaft
(DFG, German Research Foundation) within the DFG priority program SPP~2410
(project number 526031774)
and the Daimler und Benz Stiftung (Daimler and Benz foundation,
project number 32-10/22).

GB and CE were supported by the Deutsche Forschungsgemeinschaft (DFG, German Research Foundation) within the DFG priority program SPP~2410
(project number 526031774) and under Germany's Excellence Strategy EXC 2044-390685587, Mathematics Münster: Dynamics–Geometry–Structure.

%% file: 08_appendix.tex
In this appendix, we provide the more technical proofs to prepare Theorem~\ref{theorem:operatornorm_estimate_lin_adv}.

\subsection{Proof of Lemma~\ref{lemma:weightss_min_quotient}}
\label{appendix_proof_inter_quad_prop1}
\begin{proof}
    Inequality \cite[eq.~(2.3.16)]{Canuto2006book} yields the weight estimates
    \begin{equation}\label{eq:weight_estimate}
        c_1p^{-1}\sqrt{1-\hx_j^2} \le \hom_j \le c_2p^{-1}\sqrt{1-\hx_j^2},
    \end{equation}
    for every $j$, where $c_1$ and $c_2$ do not depend on $j$ and $p$. The inequality $\eqref{eq:weight_estimate}$ also can be formulated for
    Gauss-Lobatto-Legendre nodes, i.e., we just have to substitute $\hx_j \rightarrow \bx_j$ and $\hom_j \rightarrow \bom_j$ everywhere.\\
    Therefore, we get by using \eqref{eq:nodes_defi_trigGL}
    \begin{align*}
        \frac{\hom_i}{\hom_j} \le \frac{c_2p^{-1}\sqrt{1-\hx_i^2}}{c_1p^{-1}\sqrt{1-\hx_j^2}} = \frac{c_2\sin(\htheta_i)}{c_1 \sin(\htheta_j)}.
    \end{align*}
    Considering the Gauss-Legendre nodes, we can use ${2\theta}/{\pi} \le \sin(\theta) \le 1$ for $\theta \in [0, \frac{\pi}{2}]$
    combined with \eqref{eq:nodes_estimate_trigGL} for $j=1$ to obtain
    \begin{equation*}
        \frac{\hom_i}{\hom_j} \le \frac{c_2}{c_1\frac{1}{p+1}}=\frac{c_2}{c_1}(p+1).
    \end{equation*}
    As the case $p=0$ is trivial, we assume that $p>0$ and can conclude the assumption with $\hat{K}_1:=2c_2/c_1$ for the Gauss-Legendre case.
    
    For the Gauss-Lobatto-Legendre weights, we can perform similar estimates with \eqref{eq:nodes_defi_trigGLL} and \eqref{eq:nodes_estimate_trigGLL}, but notice that we have to treat the lower bound of weight $\bom_0$ differently.
    In that specific case, we can easily exploit a Legendre polynomial-related formula for the weights, that yields (see e.g. \cite[eq. (2.3.11)]{Canuto2006book}, \cite[eq. (2.7.1.12)]{davis2014methods})
    \begin{equation*}
        \bom_j = \frac{2}{p(p+1)}\frac{1}{L_p(\bx_j)^2},
    \end{equation*}
    where $L_p$ is the Legendre polynomial of degree $p$. The normalization property of $L_p$ yields $L_p(1)=1$ and by the symmetry/antisymmetry we can deduce that $L_p(\bx_0)^2=1$. Therefore, we receive an explicit expression of $\bom_0$ and can estimate
    \begin{equation*}
        \frac{\bom_i}{\bom_j} \le \frac{c_{GLL2}p^{-1}\sin(\bar{\theta_i})}{\bom_j}
        \le \frac{c_{GLL2}}{\frac{2}{(p+1)}}\le \bar{K}_1p
    \end{equation*}
    with a similar, appropriate choice of $\bar{K}_1$ as in the previous case.
\end{proof}

\subsection{Proof of Lemma~\ref{lemma:nodes_min_distance}}
\label{appendix_proof_inter_quad_prop2}
\begin{proof}
    Using a trigonometric identity \cite[eq. (2.5.2.1.3)]{Bronstein1987}, we can derive for $1 \le j \le \left\lfloor {(p+1)}/{2} \right\rfloor$, that
    \begin{equation}\label{eq:nodes_estimate_first_step}
        \begin{aligned}
        \abs{\hx_j-\hx_{j-1}} &= \abs{\cos(\htheta_j)-\cos(\htheta_{j+1})} = \abs{2\sin\left(\frac{\htheta_{j+1}-\htheta_j}{2}\right)\sin\left(\frac{\htheta_{j+1}+\htheta_j}{2}\right)}\\
        &\ge \abs{\frac{2}{\pi^2}\left(\htheta^2_{j+1}-\htheta^2_j\right)},
        \end{aligned}
    \end{equation}
    where the last estimate follows from $\sin(x)\ge {2x}/{\pi}\ge0$ for $x\in[0, {\pi}/{2}]$.
    Using \eqref{eq:nodes_estimate_trigGL}, we obtain
    \begin{equation}\label{eq:GaussLegendre_node_estimate}
        \begin{aligned}
            \abs{\hx_j-\hx_{j-1}} &\ge \abs{\frac{2}{\pi^2}\left(\frac{\pi^2(j+\frac{1}{2})^2}{(p+1)^2}-\frac{\pi^2j^2}{(p+2)^2}\right)}\\
            &= \abs{2\left(\frac{(j+\frac{1}{2})^2(p+2)^2-j^2(p+1)^2}{(p+1)^2(p+2)^2}\right)}\\
            &\ge 2\frac{(j+\frac{1}{2})^2-j^2}{(p+2)^2}\ge \frac{5}{2}\frac{1}{(p+2)^2}.
        \end{aligned}
    \end{equation}
    As this estimate holds for all $\hx_j \in (-1,0)$, we now want to include the nodes in $[0,1)$.
    Because of symmetry, it suffices to consider the distance between $\hx_{\left\lfloor{(p+1)}/{2}\right\rfloor}$ and $\hx_{\left\lfloor{(p+1)}/{2}\right\rfloor+1}\ge 0$.
    With the previous derivations in \eqref{eq:nodes_estimate_first_step}, we can deduce in the same way
    \begin{align*}
        \abs{\hx_{\left\lfloor\frac{p+1}{2}\right\rfloor+1}-\hx_{\left\lfloor\frac{p+1}{2}\right\rfloor}} \ge \hx_{\left\lfloor\frac{p+1}{2}\right\rfloor}-0&\ge \frac{2}{\pi^2}\left(
            \left(\frac{\pi}{2}\right)^2-\htheta_{\left\lfloor\frac{p+1}{2}\right\rfloor}^2\right) \ge \frac{1}{2}-\frac{2}{\pi^2}\left(\frac{\frac{p}{2}\pi}{p+2}\right)^2\\
        &= \frac{1}{2}\left(1-\frac{p^2}{(p+2)^2}\right) = \frac{2p+2}{(p+2)^2}\\
        &\ge \frac{p+2}{(p+2)^2} = \frac{1}{p+2}.
    \end{align*}
    Note that ${1}/{(p+2)}\ge 2.5/(p+2)^{2}$ for $p>0$ and therefore the estimate in \eqref{eq:GaussLegendre_node_estimate} dominates.
    Therefore, we can conclude the case of Gauss-Legendre nodes by choosing $\hat{K}_2:= 5/18$, so that for $p\ge1$,
    \begin{equation*}
        \frac{\hat{K}_2}{p^3}=\frac{5}{2\cdot9p^3}\le\frac{5p^3}{2p^3(p+2)^2}\le \hx_i-\hx_j.
    \end{equation*}

    Considering now the Gauss-Lobatto-Legendre nodes $\bx_j$, we proceed as for the Gauss-Legendre nodes using \eqref{eq:nodes_estimate_trigGLL}, resulting in
    \begin{equation}\label{eq:nodes_estimate_GLL}
        \begin{aligned}
        \abs{\bx_j-\bx_{j-1}}&\ge \frac{2}{\pi^2}\left(\bar{\theta}^2_{j+1}-\bar{\theta}^2_j\right) \\
        &\ge 2(j+1^2)\left(\frac{1}{p^2}-\frac{1}{(p+1)^2}\right)\\
        & \ge 8 \frac{2p+1}{p^2(p+1)^2} \ge \frac{16}{(p+1)^3}.
        \end{aligned}
    \end{equation}
    For the difference to an $\bx_i\ge 0$, we again have
    \begin{equation*}
        \bx_{\left\lfloor\frac{p-1}{2} \right\rfloor}-0 \ge \frac{2}{\pi^2}\left(\left(\frac{\pi}{2}\right)^2-\left(\frac{\pi}{p}\right)^2\right)=\frac{p^2-4}{2p^2},
    \end{equation*}
    which does not exceed the estimate of two other consecutive nodes in \eqref{eq:nodes_estimate_GLL} (for $p\ge 3$, as this estimate is just there required).
    For the Gauss-Lobatto-Legendre nodes, we also have to consider the distance to the nodes at the interval borders $\bx_0=-1$ and $\bx_p = 1$. There, we have
    \begin{equation*}
        \bx_1-\bx_0 = 2\sin(\bar{\theta}_1)\ge \frac{2}{\pi^2}\bar{\theta}_1^2 = \frac{2}{p^2},
    \end{equation*}
    which does dominate the estimate for $p\le4$, compared to \eqref{eq:nodes_estimate_GLL}. Therefore, we find a constant $C_{GLL}$, such that for all orders of $p$ it holds
    \begin{equation*}
        \abs{\bx_i-\bx_j} \ge \frac{C_{GLL}}{\left(p+1\right)^3}.
    \end{equation*}
    With $\hat{K_2}:= C_{GLL}/8$, we can estimate analogue as for the Gauss-Legendre nodes to conclude the statement.
\end{proof}

\subsection{Proof of Lemma~\ref{lemma:deriv_interpol_estimate}}
\label{appendix_proof_prep2}
\begin{proof}
  We use the barycentric form of the derivative matrix $D$,
  which is given by \cite[eq. (3.48)]{Kopriva2009}
  \begin{align*}
    D_{ij} &= \frac{\prod_{\substack{m = 0 \\ m \neq i}}^N \tx_i - \tx_m}{\prod_{\substack{m = 0 \\ m \neq j}}^N \tx_j - \tx_m}\left(\frac{1}{\tx_i-\tx_j}\right), \quad i\neq j, \\
    D_{ii} &= -\sum\limits_{\substack{m=0\\ m \neq j}}^M D_{im}.
  \end{align*}
  By applying the negative sum property of the diagonal element, we can expand
  \begin{align*}
    \left(D\Ip\right)_{ik}&=\sum\limits_{j=0}^p D_{ij}\tl_k(\xi_j) = \sum\limits_{\substack{j=0\\ j\neq i}}^p D_{ij}\left(\tl_k(\xi_j)-\tl_k(\xi_i)\right) \\
    &= \sum\limits_{\substack{j=0\\ j\neq i}}^p D_{ij}\left(\prod\limits_{\substack{m=0\\ m\neq k}}^p \frac{1-\tx_m+\alpha(1+\tx_j)}{\tx_k-\tx_m}
    - \prod\limits_{\substack{m=0\\ m\neq k}}^p \frac{1-\tx_m+\alpha(1+\tx_i)}{\tx_k-\tx_m} \right).
  \end{align*}
  In the last line, we see that the products just differ by the part that depends on $\alpha$. We can therefore extract this term, such that
  \begin{align*}
    \prod\limits_{\substack{m=0\\ m\neq k}}^p& \frac{1-\tx_m+\alpha(1+\tx_i)}{\tx_k-\tx_m}\\
    &= \frac{\left(\prod\limits_{\substack{m=0\\ m\neq k}}^p 1-\tx_m\right) + \alpha(1+\tx_i)\left(\sum\limits_{\substack{m_1=0\\ m_1\neq k}}^p\left(\prod\limits_{\substack{m=0\\ m\neq m_1,k}}^p 1-\tx_{m}\right)\right)
     + \alpha^2(1+\tx_i)^2\left(\sum\limits_{\substack{m_1,m_2=0\\ m_1, m_2\neq k}}^p \left(\prod\limits_{\substack{m=0\\ m\neq m_1, m_2,k}}^p1-\tx_m\right)\right) + \hdots}{\prod\limits_{\substack{m=0\\ m\neq k}}^p(\tx_k-\tx_m)}.
  \end{align*}
  Therefore, we can deduce that
  \begin{equation*}
    (D\Ip)_{ik} = \sum\limits_{\substack{j=0 \\ j\neq i}}^p D_{ij}\left(\prod\limits_{\substack{m=0 \\ m \neq k}}^p\frac{1-\tx_m}{\tx_k-\tx_m}-\prod\limits_{\substack{m=0 \\ m \neq k}}^p\frac{1-\tx_m}{\tx_k-\tx_m} + \alpha \mathcal{P}(\alpha)\right)
    =  \alpha \sum\limits_{\substack{j=0 \\ j\neq i}}^pD_{ij}\mathcal{P}(\alpha),
  \end{equation*}
  where $\mathcal{P}(\alpha)$ is a polynomial in $\alpha$.
  Knowing this notation about $(D\Ip)_{ik}$, we can do similar estimates about
  $D_{jk}\mathcal{P}(\alpha)$ as in Lemma~\ref{lemma:interpol_estimate} (which we will not explore again in detail here),
  that will again lead to an exponential scaling in $p$ for larger cut cells.
  As $\mathcal{P}(\alpha)$ is bounded for all $\alpha \in [0, 0.5]$, we can omit this dependency.
  Nevertheless, we remain with the exponential scaling in $p$, and using entry-based matrix norms (see, e.g., \cite[eq. (2.3.8)]{MatrixComputations})
  we can estimate the whole expression by a function $\mathcal{C}(p)$.
\end{proof}

\subsection{Proof of Lemma~\ref{lemma:interpol_estimate_outflowbound}}
\label{appendix_proof_prep3}
\begin{proof}
  In terms of interpolation and extrapolation operators, we can write $\hb_{J_2^0}=\hl^TB_{J}$, where
  \begin{equation*}
    B_J= \left(\tl_1(1+2\alpha), \tl_2(1+2\alpha), \hdots, \tl_{p+1}(1+2\alpha)\right)
  \end{equation*}
  is the interpolation matrix from the respective nodes of the reference element $[-1, 1]$ to the reference outflow cell boundary $1+2\alpha$ (that is limited by $1+2\lambda_c$).
  Therefore, we receive with the extended outflow boundary value $\bar{u}:=B_Ju$
  \begin{align*}
    \|\hb_{J_2^0}u\|_{M_R} &= u^TB_J^TLM_RL^TB_Ju \\
    &=\bar{u}LM_RL^T\bar{u} =\bar{u}^2\sum\limits_{j=0}^{p+1}\omega_j\hat{L_j}^2 = \frac{C_2}{4}\bar{u}^2,
  \end{align*}
  where $\omega_j$ denotes the respective GL or GLL weights. To estimate $\bar{u}$, which is exactly the extrapolation to the outflow boundary of the cut cell, we follow the same arguments as in the proof of Lemma~\ref{lemma:interpol_estimate}
  under consideration that an additional factor $1/\max_{i=1, \hdots, p+1}(\omega_i)$ is required that we include in $C$. This proves the statement.
\end{proof}